\documentclass[10pt]{article}

\usepackage{mathrsfs,amssymb,amsmath,amsfonts}
\usepackage{amsthm}
\usepackage{bbm}
\usepackage{tipa}
\usepackage{enumerate}
\usepackage{color}
\usepackage[titletoc,title]{appendix}
\usepackage{cite}
\usepackage[hyphens]{url}
\usepackage[bookmarks=false,hidelinks]{hyperref}

\usepackage{arydshln}

\newtheorem{theorem}{Theorem}[section]
\newtheorem{lemma}[theorem]{Lemma}

\newtheorem{proposition}[theorem]{Proposition}

\newtheorem{remark}[theorem]{Remark}

\theoremstyle{definition}
\newtheorem{example}[theorem]{Example}

\newenvironment{keywords}{{\bf Key words: }}{}
\newenvironment{AMS}{{\bf AMS subject classification: }}{}

\newcommand{\esssup}{\mathop{\mathrm{esssup}}}
\newcommand{\essinf}{\mathop{\mathrm{essinf}}}

\numberwithin{equation}{section}

\newcommand{\udots}{\mathinner{\mskip1mu\raise1pt\vbox{\kern7pt\hbox{.}}
\mskip2mu\raise4pt\hbox{.}\mskip2mu\raise7pt\hbox{.}\mskip1mu}}

\begin{document}

\title{Mean-Field Type FBSDEs under Domination-Monotonicity Conditions and Application to LQ Problems
\footnote{This work is supported in part by the National Key R\&D Program of China (2018YFA0703900) and the National Natural Science
Foundation of China (11871310).}}

\author{Ran Tian \footnote{Zhongtai Securities Institute for Financial Studies, Shandong University, Jinan, Shandong, 250100, China. email: {\tt tianyiran@mail.sdu.edu.cn}}  
and  
Zhiyong Yu \footnote{Corresponding author. School of Mathematics, Shandong University, Jinan, Shandong, 250100, China. email: {\tt yuzhiyong@sdu.edu.cn}. }}

\maketitle

\begin{abstract}
This paper is concerned with a class of mean-field type coupled forward-backward stochastic differential equations (MF-FBSDEs, for short), in which the coupling appears in integral terms, terminal terms, and initial terms. Inspired by various mean-field type linear-quadratic (MF-LQ, for short) optimal control problems, we proposed a type of randomized domination-monotonicity conditions, under which and the usual Lipschitz condition, we obtain a well-posedness result on MF-FBSDEs in the sense of square integrability including the unique solvability, an estimate of the solution, and the related continuous dependence property of the solution on the coefficients. The result of MF-FBSDEs in turn extends MF-LQ problems in the literature to a general situation where the initial states or the terminal states are also controlled at the same time, and gives explicit expressions of the related unique optimal controls. 
\end{abstract}

\begin{keywords}
forward-backward stochastic differential equation, stochastic linear-quadratic problem, stochastic optimal control, mean-field, domination-monotonicity condition
\end{keywords}\\

\begin{AMS}
93E20, 60H10, 49N10
\end{AMS}

\section{Introduction}

Historically, as a kind of mean-field type stochastic differential equations (MF-SDE, for short), the research on McKean-Vlasov SDEs started from the pioneering works of Kac \cite{Kac-56} and McKean \cite{McK-67}. In the past fifteen years, accompanying the research boom of mean-field games led by Huang et al. \cite{HMC-06} and Lasry and Lions \cite{LL-07}, the MF-SDEs gained a rapid development. Especially, mean-field type backward SDEs (MF-BSDEs, for short), which admit a typical structure different from the corresponding forward equations, were introduced and studied by Buckdahn et al. \cite{BDLP-09, BLP-09}. Due to the requirements of both theory and applications, a (forward) MF-SDE and an MF-BSDE were coupled together to form a system called an MF-FBSDE, and were studied by many scholars; see Carmona and Delarue \cite{CD-13, CD-15}, Bensoussan et al. \cite{BYZh-15}, Wei et al. \cite{WYY-19}, and so on.

In this paper, we consider the following MF-FBSDE:
\begin{equation}\label{1:FBSDE}
\left\{
\begin{aligned}
& \mathrm dx(s) =b\big( s,\theta(s),\mathbb E_t[\theta(s)] \big)\, \mathrm ds +\sum_{i=1}^d \sigma_i \big( s,\theta(s),\mathbb E_t[\theta(s)] \big)\, \mathrm dW_i(s),\quad s\in [t,T],\\
& \mathrm dy(s) = g\big( s,\theta(s),\mathbb E_t[\theta(s)] \big)\, \mathrm ds +\sum_{i=1}^d z_i(s)\, \mathrm dW_i(s),\quad s\in [t,T],\\
& x(t) = \Psi\big(y(t)\big),\qquad y(T) = \Phi\big(x(T), \mathbb E_t[x(T)]\big),
\end{aligned}
\right.
\end{equation}
where $\mathbb E_t[\cdot] \equiv \mathbb E[\cdot | \mathcal F_t]$ is the conditional expectation with respect to $\mathcal F_t$, $\theta(\cdot) := (x(\cdot)^\top,y(\cdot)^\top,z(\cdot)^\top)^\top$ with $z(\cdot) = (z_1(\cdot)^\top,z_2(\cdot)^\top,\dots,z_d(\cdot)^\top)^\top$ is the unknown process, and $W(\cdot) := (W_1(\cdot),W_2(\cdot),\dots,W_d(\cdot))^\top$ is a $d$-dimensional Brownian motion. Here and hereafter, the superscript $\top$ denotes the transpose of a vector or a matrix. For convenience, we also denote $\Gamma(\cdot) := (g(\cdot)^\top, b(\cdot)^\top, \sigma(\cdot)^\top)^\top$ and $\sigma(\cdot) := (\sigma_1(\cdot)^\top,\sigma_2(\cdot)^\top,\dots,\sigma_d(\cdot)^\top)^\top$. Then, all of the coefficients of MF-FBSDE \eqref{1:FBSDE} are collected by $(\Psi, \Phi, \Gamma)$.

To our best knowledge, in the literature on MF-FBSDEs, the coupling between forward variable $x(\cdot)$ and backward variables $(y(\cdot),z(\cdot))$ only appeared in the integral coefficient $\Gamma(\cdot)$ and the terminal coefficient $\Phi(\cdot)$, while the initial coefficient $\Psi(\cdot)$ was always in the decoupled form, i.e., $\Psi(\cdot) \equiv x_t$ which is a given $\mathcal F_t$-measurable random variable. In comparison, the coupling occurs in all three coefficients $(\Psi,\Phi,\Gamma)$ of \eqref{1:FBSDE}. The present paper is the first time to consider this general situation.

It is well known that, if only the uniformly Lipschitz continuity of the coefficients is proposed, MF-FBSDE \eqref{1:FBSDE} (even in the absence of initial coupling and/or mean-field terms) may be unsolvable on an arbitrarily large time interval (see a counterexample in Antonelli \cite{A-93}).  Carmona and Delarue \cite{CD-13} introduced some assumptions including the uniform ellipticity and the boundedness of the coefficient $\sigma(\cdot)$, and for the first time obtained an existence result on large intervals for coupled MF-FBSDEs. In their second paper \cite{CD-15}, they introduced the convexity assumption to get a unique solvability result. Moreover, Bensoussan et al. \cite{BYZh-15} proposed a kind of monotonicity conditions and improve the result in \cite{CD-15}. In the present paper, we will introduce a type of {\it domination-monotonicity conditions} to ensure the global solvability of MF-FBSDEs. Such conditions can be regarded as a generalization of monotonicity conditions in \cite{BYZh-15}. We also note that a special and linear case was studied in Wei et al. \cite{WYY-19}.

The domination-monotonicity conditions are actually rooted in various MF-LQ optimal control problems; see Yong \cite{Yong-13}, Ni et al. \cite{NLZh-16}, Sun \cite{S-17}, Li et al. \cite{LSX-19}, Wang \cite{W-19}, and Li et al. \cite{LLY-20}, for example. Let us explain this point in detail. When we apply the Pontryagin type maximum principle approach to study the necessary condition of an MF-LQ problem, by introducing an adjoint equation, the optimal control will be closely linked with a mean-field type Hamiltonian system (which is actually an MF-FBSDE, see \eqref{4.1:Ham} and \eqref{4.2:Ham}). Moreover, by imposing some  uniformly positive definiteness condition (see Condition (MF-FLQ-PD) and Condition (MF-BLQ-PD) in Section \ref{Sec:LQ}), the necessary condition is also sufficient. Taking MF-LQ problems as an understanding background and intuitive guidance, we propose the domination-monotonicity conditions for  the nonlinear MF-FBSDE \eqref{1:FBSDE}. The features and novelties of such conditions are summarized as the following three points:
\begin{enumerate}[(i)]
\item The domination-monotonicity conditions (see Assumption (H4) and Remark \ref{3:Rem:Sym}) accurately correspond to the expression forms of optimal controls and the related uniformly positive definiteness conditions in the four classes of MF-LQ problems (see Proposition \ref{4.1:Prop}, Proposition \ref{4.2:Prop} and Remark \ref{4.2:Rem}).

\item Compared with monotonicity conditions in \cite{BYZh-15}, the introduction of matrices and matrix-valued functions in our domination-monotonicity conditions increases flexibility and extensiveness. In detail, due to their introduction, on the one hand, a new kind of domination conditions are created which strengthen the Lipschitiz condition; on the other hand, at the same time it also weakens the monotonicity conditions in \cite{BYZh-15}. The flexible selection of matrices and matrix-valued functions provides us with more possibilities.

\item Compared with the classical case without involving mean-field terms (see \cite{Y-21+}), the domination-monotonicity conditions in this paper present a randomized form (see \eqref{3:M:Phi} and \eqref{3:M:Gamma}). This is a response to the difficulty caused by the mean-field terms. We notice that this randomization technique has been used in \cite{BYZh-15}. 
\end{enumerate}

In this paper, we will develop the {\it method of continuation}, which is originally introduced by Hu and Peng \cite{HP-95}, Yong \cite{Yong-97}, and Peng and Wu \cite{PW-99} when they studied the classical coupled FBSDEs without mean-field terms, and obtain the well-posedness of MF-FBSDE \eqref{1:FBSDE} in the sense of square integrability including the existence and uniqueness of the solution, an estimate, and the related continuous dependence property of the solution on the coefficients (see Theorem \ref{3:THM:FBSDE}). In order to overcome the technical difficulty brought with the mean-field terms, besides the randomization technique stated in the above Point (iii), we will also employ a decomposition technique of random variables: $\xi = (\xi -\mathbb E_t[\xi]) +\mathbb E_t[\xi]$ which is originally introduced by Yong \cite{Yong-13}. We notice that although this decomposition is simple, it has the advantage that its first part has zero conditional expectation and the second part is $\mathcal F_t$-measurable. Due to this, it is often adopted to deal with many mean-field type problems (see Lin et al. \cite{LJZh-19} and Tian et al. \cite{TYZh-20}, for example). But, the difficulty has still not been completely overcome. Then, we have to assume that the matrices and matrix-valued processes involved in the domination-monotonicity conditions are deterministic. We are finally able to overcome the difficulty and complete all analyses and proofs in this paper.

As an application of the unique solvability result of MF-FBSDE \eqref{1:FBSDE}, we re-examined several MF-LQ problems. It is worth noting that, in MF-LQ problems in the literature, control is only applied to the drift and diffusion terms of the controlled systems; while in the present paper, corresponding to the feature that the couplings appear in both $\Psi(\cdot)$ and $\Phi(\cdot)$ of MF-FBSDE \eqref{1:FBSDE}, control can also be applied to the initial values of forward MF-LQ problems (see the forward controlled system \eqref{4.1:Sys}, for example) and the terminal values of backward MF-LQ problems (see the backward controlled system \eqref{4.2:Sys}) at the same time. These extensions are interesting in their own right. As conclusions, the unique solvability of MF-FBSDE \eqref{1:FBSDE} under domination-monotonicity conditions implies that of Hamiltonian systems arising from these generalized MF-LQ problems, and further implies the existence and uniqueness of optimal controls.

The rest of this paper is organized as follows. In Section \ref{Sec:Pre}, we define some notations and give the preliminaries about MF-SDEs and MF-BSDEs. In Section \ref{Sec:FBSDE}, we rigorously state the domination-monotonicity conditions in detail, then under them, we prove the well-posedness of MF-FBSDE \eqref{1:FBSDE}. In Section \ref{Sec:LQ}, we apply the results of MF-FBSDEs obtained in the previous section to study some generalized MF-LQ problems, and obtain the explicit expressions of the unique optimal controls based on the related mean-field type Hamiltonian systems. Some proofs for the results of MF-SDEs and MF-BSDEs are put in Appendix \ref{Sec:App}.

\section{Notations and preliminaries}\label{Sec:Pre}

Let $\mathbb R^n$ be the $n$-dimensional Euclidean space equipped with the Euclidean inner product $\langle \cdot,\ \cdot \rangle$. The induced Euclidean norm is denoted by $|\cdot|$. Let $\mathbb R^{m\times n}$ be the  set of all $(m\times n)$ matrices and $\mathbb S^n \subset \mathbb R^{n\times n}$ be the subset consisting of all $(n\times n)$ symmetric matrices. In this paper, we use the operator norm of matrices:
\[
\Vert A \Vert := \sup_{\mathbb R^n \ni x \neq 0} \frac{|Ax|}{|x|} \quad \mbox{for any } A\in \mathbb R^{m\times n}.
\]

Let $0<T<\infty$ be a fixed time horizon. Let $(\Omega,\mathcal F,\mathbb F,\mathbb P)$ be a complete filtered probability space on which a $d$-dimensional standard Brownian motion $W(\cdot)$ is defined and $\mathbb F = \{ \mathcal F_s,\ 0\leq s\leq T \}$ is the natural filtration of $W(\cdot)$ augmented by all $\mathbb P$-null sets. We also set $\mathcal F=\mathcal F_T$.

Let $0\leq t <T$ and $q=1,2$. We introduce some Banach (or Hilbert in case of $L^2_{\mathcal F_t}(\Omega;\mathbb R^n)$ or $L^2_{\mathbb F}(t,T;\mathbb R^n)$) spaces as follows:
\begin{itemize}
\item $L^2_{\mathcal F_t}(\Omega;\mathbb R^n)$ is the set of all $\mathcal F_t$-measurable random variables $\xi:\Omega\rightarrow\mathbb R^n$ such that
\[
\Vert \xi \Vert_{L^2_{\mathcal F_t}(\Omega;\mathbb R^n)} := \big\{\mathbb E\big[|\xi|^2\big]\big\}^{1/2} <\infty.
\]
\item $L^2_{\mathbb F}(\Omega; L^q([t,T];\mathbb R^n))$ is the set of all $\mathbb F$-progressively measurable processes $f:[t,T] \times \Omega \rightarrow \mathbb R^n$ such that
\[
\Vert f(\cdot) \Vert_{L^2_{\mathbb F}(\Omega; L^q([t,T];\mathbb R^n))} := \bigg\{ \mathbb E \bigg[ \bigg( \int_t^T |f(s)|^q\, \mathrm ds \bigg)^{2/q} \bigg] \bigg\}^{1/2} <\infty.
\]
When $q=1$, we denote $L^2_{\mathbb F}(\Omega; L([t,T];\mathbb R^n)) =L^2_{\mathbb F}(\Omega; L^1([t,T];\mathbb R^n))$. When $q=2$, we denote $L^2_{\mathbb F}(t,T;\mathbb R^n) = L^2_{\mathbb F}(\Omega; L^2([t,T];\mathbb R^n))$.
\item $L^2_{\mathbb F}(\Omega;C([t,T];\mathbb R^n))$ is the set of all $\mathbb F$-progressively measurable processes $f:[t,T] \times \Omega \rightarrow \mathbb R^n$ such that for almost all $\omega\in \Omega$, $s\mapsto f(s,\omega)$ is continuous and
\[
\Vert f(\cdot) \Vert_{L^2_{\mathbb F}(\Omega;C([t,T];\mathbb R^n))} := \bigg\{ \mathbb E\bigg[ \sup_{s\in [t,T]} |f(s)|^2 \bigg] \bigg\}^{1/2} <\infty.
\] 
\end{itemize}

Now, we turn to consider an MF-SDE:
\begin{equation}\label{2:SDE}
\left\{
\begin{aligned}
& \mathrm dx(s) = b\big( s, x(s), \mathbb E_t[x(s)] \big)\, \mathrm ds +\sum_{i=1}^d \sigma_i\big( s, x(s), \mathbb E_t[x(s)] \big)\, \mathrm dW_i(s),\quad s\in [t,T],\\
& x(t) =x_t.
\end{aligned}
\right.
\end{equation}
For convenience, we denote $\sigma(\cdot):= ( \sigma_1(\cdot)^\top, \sigma_2(\cdot)^\top,\dots,\sigma_d(\cdot)^\top )^\top$. The coefficients $(x_t, b,\sigma)$ are assumed to satisfy the following

\medskip

\noindent{\bf Assumption (H1).} (i) For any $x,x'\in \mathbb R^n$, the processes $b(\cdot,x,x')$ and $\sigma(\cdot,x,x')$ are $\mathbb F$-progressively measurable. Moreover, $x_t \in L^2_{\mathcal F_t}(\Omega;\mathbb R^n)$, $b(\cdot,0,0) \in L^2_{\mathbb F}(\Omega;L([t,T];\mathbb R^n))$ and $\sigma(\cdot,0,0) \in L^2_{\mathbb F}(t,T;\mathbb R^{nd})$.

\smallskip

\noindent (ii) The mappings $b$ and $\sigma$ are uniformly Lipschitz continuous with respect to $(x,x')$, i.e., there exists a constant $L>0$ such that 
\[
\big|b(s,x,x') -b(s,\bar x,\bar x')\big| +\big|\sigma(s,x,x') -\sigma(s,\bar x,\bar x')\big| \leq L\big( | x-\bar x | +| x'-\bar x' |\big),
\]
for any $x,\bar x,x',\bar x'\in \mathbb R^n$ and almost all $(s,\omega) \in [t,T] \times \Omega$.

\medskip

We have the following result.

\begin{proposition}\label{2:Prop:SDE}
Under Assumption (H1), MF-SDE \eqref{2:SDE} with coefficients $(x_t,b,\sigma)$ admits a unique solution $x(\cdot) \in L^2_{\mathbb F}(\Omega;C([t,T];\mathbb R^n))$. Moreover, we have the following estimate:
\begin{equation}\label{2:SDE:Est1}
\mathbb E_t \bigg[ \sup_{s\in [t,T]} |x(s)|^2  \bigg] \leq K \bigg\{ |x_t|^2 +\mathbb E_t\bigg[ \bigg( \int_t^T \big|b(s,0,0)\big|\, \mathrm ds \bigg)^2 +\int_t^T \big|\sigma(s,0,0)\big|^2\, \mathrm ds \bigg]\bigg\},
\end{equation}
where $K:=K(T-t,L) >0$ is a constant depending on $(T-t)$ and the Lipschitz constant of the mappings $b$ and $\sigma$. Furthermore, let $(\bar x_t,\bar b,\bar \sigma)$ be another set of coefficients, and assume that $\bar x(\cdot) \in L^2_{\mathbb F}(\Omega;C([t,T];\mathbb R^n))$ is a solution to the MF-SDE with $(\bar x_t,\bar b,\bar \sigma)$. We continue to assume that $\bar x_t\in L^2_{\mathcal F_t}(\Omega;\mathbb R^n)$, $\bar b(\cdot,\bar x(\cdot),\mathbb E_t[\bar x(\cdot)]) \in L^2_{\mathbb F}(\Omega;L([t,T];\mathbb R^n))$ and $\bar\sigma (\cdot,\bar x(\cdot),\mathbb E_t[\bar x(\cdot)]) \in L^2_{\mathbb F}(t,T;\mathbb R^{nd})$. Then,
\begin{equation}\label{2:SDE:Est2}
\begin{aligned}
& \mathbb E_t \bigg[ \sup_{s\in [t,T]} |x(s) -\bar x(s)|^2 \bigg] \leq K \bigg\{ |x_t-\bar x_t|^2 +\mathbb E_t \bigg[ \bigg( \int_t^T \Big| b\big( s, \bar x(s), \mathbb E_t [\bar x(s)] \big)\\
& -\bar b\big( s, \bar x(s), \mathbb E_t [\bar x(s)] \big) \Big|\, \mathrm ds \bigg)^2 +\int_t^T \Big| \sigma \big( s, \bar x(s), \mathbb E_t [\bar x(s)] \big) -\bar\sigma \big( s, \bar x(s), \mathbb E_t [\bar x(s)] \big) \Big|^2\, \mathrm ds \bigg] \bigg\},
\end{aligned}
\end{equation}
where $K$ is the same constant as in \eqref{2:SDE:Est1}.
\end{proposition}

The results in Proposition \ref{2:Prop:SDE} are standard. We believe that they are not new. However, we are not able to find an exact reference, then we give a proof and put it in Appendix \ref{Sub:App1} for the readers' convenience.

Next, we consider an MF-BSDE as follows:
\begin{equation}\label{2:BSDE}
\left\{
\begin{aligned}
& \mathrm dy(s) = g \big( s,y(s),\mathbb E_t[y(s)], z(s), \mathbb E_t[z(s)] \big)\, \mathrm ds +\sum_{i=1}^d z_i(s)\, \mathrm dW_i(s),\quad s\in [t,T],\\
& y(T) =y_T,
\end{aligned}
\right.
\end{equation}
where $z(\cdot) := (z_1(\cdot)^\top, z_2(\cdot)^\top, \dots, z_d(\cdot)^\top)^\top$. For the coefficients $(y_T, g)$ of MF-BSDE \eqref{2:BSDE}, we introduce the following

\medskip

\noindent {\bf Assumption (H2).} (i) For any $y, y' \in \mathbb R^n$ and $z, z' \in \mathbb R^{nd}$, the process $g(\cdot,y,y',z,z')$ is $\mathbb F$-progressively measurable. Moreover, $y_T\in L^2_{\mathcal F_T}(\Omega;\mathbb R^n)$ and $g(\cdot,0,0,0,0) \in L^2_{\mathbb F}(\Omega;L([t,T];\mathbb R^n))$.

\smallskip

\noindent (ii) The mapping $g$ is uniformly Lipschitz continuous with respect to $(y,y',z,z')$, i.e., there exists a constant $L>0$ such that 
\[
\big|g\big( s,y,y',z,z' \big) -g\big( s,\bar y,\bar y',\bar z,\bar z' \big)\big| \leq L\big( | y-\bar y | +| y'-\bar y' | +| z-\bar z | +| z'-\bar z' | \big).
\]
for any $y,\bar y, y', \bar y'\in \mathbb R^n$, any $z,\bar z,z',\bar z' \in \mathbb R^{nd}$ and almost all $(s,\omega) \in [t,T]\times\Omega$.

\medskip

\begin{proposition}\label{2:Prop:BSDE}
Under Assumption (H2), MF-BSDE \eqref{2:BSDE} with coefficients $(y_T,g)$ admits a unique solution $(y(\cdot),z(\cdot)) \in L^2_{\mathbb F}(\Omega;C([t,T];\mathbb R^n)) \times L^2_{\mathbb F}(t,T;\mathbb R^{nd})$. Moreover, we have the following estimate:
\begin{equation}\label{2:BSDE:Est1}
\mathbb E_t \bigg[ \sup_{s\in [t,T]} |y(s)|^2 +\int_t^T |z(s)|^2\, \mathrm ds \bigg] \leq K \mathbb E_t\bigg[ |y_T|^2 +\bigg( \int_t^T \big| g\big( s,0,0,0,0 \big)  \big|\, \mathrm ds \bigg)^2 \bigg],
\end{equation}
where $K := K(T-t,L)>0$ is a constant depending on $(T-t)$ and the Lipschitz constant of the mapping $g$. Furthermore, let $(\bar y_t, \bar g)$ be another set of coefficients, and assume that $(\bar y(\cdot),\bar z(\cdot))\in L^2_{\mathbb F}(\Omega;C([t,T];\mathbb R^n)) \times L^2_{\mathbb F}(t,T;\mathbb R^{nd})$ is a solution to the MF-BSDE with $(\bar y_T, \bar g)$. We continue to assume that $\bar y_T\in L^2_{\mathcal F_T}(\Omega;\mathbb R^n)$ and $\bar g(\cdot,\bar y(\cdot),\mathbb E_t[\bar y(\cdot)],\bar z(\cdot),\mathbb E_t[\bar z(\cdot)]) \in L^2_{\mathbb F}(\Omega;L([t,T];\mathbb R^n))$. Then,
\begin{equation}\label{2:BSDE:Est2}
\begin{aligned}
& \mathbb E_t \bigg[ \sup_{s\in [t,T]} \big| y(s) -\bar y(s) \big|^2 +\int_t^T \big| z(s) -\bar z(s) \big|^2\, \mathrm ds \bigg] \leq K \mathbb E_t \bigg\{ \big| y_T -\bar y_T \big|^2\\
& +\bigg( \int_t^T \Big| g\big( s,\bar y(s), \mathbb E_t[\bar y(s)], \bar z(s), \mathbb E_t[\bar z(s)] \big) -\bar g\big( s,\bar y(s), \mathbb E_t[\bar y(s)], \bar z(s), \mathbb E_t[\bar z(s)] \big) \Big|\, \mathrm ds \bigg)^2 \bigg\},
\end{aligned}
\end{equation}
where $K$ is the same constant as in \eqref{2:BSDE:Est1}.
\end{proposition}

The proof of the above Proposition \ref{2:Prop:BSDE} can be found in Appendix \ref{Sub:App2}.

At the end of this section, we continue to introduce some other spaces which will be used in our following analysis. 
\begin{itemize}
\item $L^\infty(t,T;\mathbb R^{m\times n})$ is the set of all Lebesgue measurable functions $A:[t,T]\rightarrow \mathbb R^{m\times n}$ such that
\[
\Vert A(\cdot) \Vert_{L^\infty(t,T;\mathbb R^{m\times n})} := \esssup_{s\in [t,T]} \Vert A(s) \Vert <\infty.
\]
\end{itemize}
Some product spaces are also introduced: 
\begin{itemize}
\item $M^2_{\mathbb F}(t,T;\mathbb R^{n(2+d)}) := L^2_{\mathbb F}(\Omega;C([t,T];\mathbb R^n)) \times L^2_{\mathbb F}(\Omega;C([t,T];\mathbb R^n)) \times L^2_{\mathbb F}(t,T;\mathbb R^{nd})$. For any $\theta(\cdot) =(x(\cdot)^\top,y(\cdot)^\top,z(\cdot)^\top) \in M^2_{\mathbb F}(t,T;\mathbb R^{n(2+d)})$, its norm is given by
\[
\Vert \theta(\cdot) \Vert_{M^2_{\mathbb F}(t,T;\mathbb R^{n(2+d)})} := \bigg\{ \mathbb E\bigg[ \sup_{s\in [t,T]} |x(s)|^2 +\sup_{s\in [t,T]} |y(s)|^2 +\int_t^T |z(s)|^2\, \mathrm ds \bigg] \bigg\}^{1/2}.
\]
\item $\mathcal M^2_{\mathbb F}(t,T;\mathbb R^{n(2+d)}) := L^2_{\mathbb F}(\Omega;L([t,T];\mathbb R^n)) \times L^2_{\mathbb F}(\Omega;L([t,T];\mathbb R^n)) \times L^2_{\mathbb F}(t,T;\mathbb R^{nd})$. For any $\rho(\cdot) =(\varphi(\cdot)^\top, \psi(\cdot)^\top, \gamma(\cdot)^\top)^\top$, its norm is 
\[
\Vert \rho(\cdot) \Vert_{\mathcal M^2_{\mathbb F}(t,T;\mathbb R^{n(2+d)})} := \bigg\{ \mathbb E\bigg[ \bigg( \int_t^T |\varphi(s)|\, \mathrm ds \bigg)^2 +\bigg( \int_t^T |\psi(s)|\, \mathrm ds \bigg)^2 +\int_t^T |\gamma(s)|^2\, \mathrm ds \bigg] \bigg\}^{1/2}.
\]
\item $\mathcal H[t,T] := L^2_{\mathcal F_t}(\Omega;\mathbb R^n) \times L^2_{\mathcal F_T}(\Omega;\mathbb R^n) \times \mathcal M^2_{\mathbb F}(t,T;\mathbb R^{n(2+d)})$. For any $(\xi, \eta, \rho(\cdot)) \in \mathcal H[t,T]$, its norm is given by
\[
\Vert (\xi, \eta, \rho(\cdot)) \Vert_{\mathcal H[t,T]} := \Big\{ \Vert \xi \Vert_{L^2_{\mathcal F_t}(\Omega;\mathbb R^n)}^2 +\Vert \eta \Vert_{L^2_{\mathcal F_T}(\Omega;\mathbb R^n)}^2 +\Vert \rho(\cdot) \Vert_{\mathcal M^2_{\mathbb F}(t,T;\mathbb R^{n(2+d)})}^2 \Big\}^{1/2}.
\]
\end{itemize}

\section{MF-FBSDEs with domination-monotonicity conditions}\label{Sec:FBSDE}

In this section, we devote ourselves to investigating MF-FBSDE \eqref{1:FBSDE}. Similar to MF-SDEs and MF-BSDEs, we introduce the following assumptions on the coefficients $(\Psi,\Phi,\Gamma)$:

\medskip

\noindent {\bf Assumption (H3).} (i) For any $y\in \mathbb R^n$, $\Psi(y)$ is $\mathcal F_t$-measurable. For any $x, x' \in \mathbb R^n$, $\Phi(x,x')$ is $\mathcal F_T$-measurable. For any $\theta, \theta' \in \mathbb R^{n(2+d)}$, $\Gamma(\cdot,\theta,\theta')$ is $\mathbb F$-progressively measurable. Moreover, $(\Psi(0), \Phi(0,0), \Gamma(\cdot,0,0))\in \mathcal H[t,T]$.

\smallskip

\noindent (ii) The mappings $\Psi$, $\Phi$ and $\Gamma$ are uniformly Lipschitz continuous with respect to $y$, $(x,x')$ and $(\theta,\theta')$ respectively, i.e., the exists a constant $L>0$ such that
\[
\left\{
\begin{aligned}
& \big| \Psi(y) -\Psi(\bar y) \big| \leq L |y-\bar y|,\\
& \big| \Phi( x, x' ) -\Phi( \bar x, \bar x' ) \big| \leq L\big( |x-\bar x| +|x'-\bar x'| \big),\\
& \big| \Gamma(s,\theta,\theta') -\Gamma(s,\bar\theta,\bar\theta') \big| \leq L\big( |\theta-\bar\theta| +|\theta'-\bar\theta'| \big),
\end{aligned}
\right.
\]
for any $y,\bar y,x,\bar x,x',\bar x' \in \mathbb R^n$, any $\theta, \bar\theta, \theta', \bar\theta' \in \mathbb R^{n(2+d)}$ and almost all $(s,\omega) \in [t,T] \times \Omega$.

\medskip

Besides the above Assumption (H3), the following domination-monotonicity conditions are also imposed:

\medskip

\noindent{\bf Assumption (H4).} There exist two constants $\mu\geq 0$, $\nu\geq 0$, three matrices $H\in \mathbb R^{m_1\times n}$, $P, \widetilde P \in \mathbb R^{m_2\times n}$, and six matrix-valued processes $A(\cdot), \widetilde A(\cdot), B(\cdot), \widetilde B(\cdot) \in L^\infty(t,T;\mathbb R^{m_3\times n})$, $C(\cdot) = (C_1(\cdot),C_2(\cdot),\dots,C_d(\cdot))$, $\widetilde C(\cdot) =(\widetilde C_1(\cdot), \widetilde C_2(\cdot),\dots, \widetilde C_d(\cdot))$ with $C_i(\cdot), \widetilde C_i(\cdot) \in L^\infty(t,T;\mathbb R^{m_3\times n})$ (where $m_1, m_2, m_3 \in \mathbb N$ are given) such that the following conditions hold:

\smallskip

\noindent (i) One of the following two cases holds. {\it Case A}: $\mu>0$ and $\nu=0$. {\it Case B}: $\mu =0$ and $\nu>0$.

\smallskip

\noindent (ii) (Domination conditions). For any $x, \bar x, x', \bar x', y, \bar y, y', \bar y' \in \mathbb R^n$, any $z, \bar z, z', \bar z' \in \mathbb R^{nd}$, and almost all $(s,\omega) \in [t,T]\times \Omega$, with the notations $\widehat x =x-\bar x$, $\widehat x' =x' -\bar x'$, $\widehat y =y-\bar y$, $\widehat y' =y' -\bar y'$, $\widehat z =z-\bar z$, $\widehat z' =z' -\bar z'$, and $f=b,\sigma$, we have 
\begin{equation}\label{3:D}
\left\{
\begin{aligned}
& |\Psi(y) -\Psi(\bar y)| \leq \frac 1 \mu | H\widehat y |,\\
& \big|\Phi(x,x') -\Phi(\bar x, \bar x')\big| \leq \frac 1 \nu \left| \begin{pmatrix} P(\widehat x -\widehat x')\\ \widetilde P \widehat x' \end{pmatrix} \right|,\\
& \big|g(s,x,x',y,y',z,z') -g(s,\bar x,\bar x',y,y',z,z')\big| \leq \frac 1 \nu \left| \begin{pmatrix} A(s) (\widehat x -\widehat x')\\ \widetilde A(s) \widehat x' \end{pmatrix} \right|,\\
& \big| f(s,x,x',y,y',z,z') -f(s,x,x',\bar y,\bar y',\bar z,\bar z') \big| \leq \frac 1 \mu \left| \begin{pmatrix}
B(s)(\widehat y -\widehat y') +C(s)(\widehat z -\widehat z')\\
\widetilde B(s) \widehat y' +\widetilde C(s) \widehat z'
\end{pmatrix} \right|.
\end{aligned}
\right.
\end{equation}
Here, we have a bit of abusive notations, i.e., when $\mu =0$ (resp. $\nu=0$), $1/\mu$ (resp. $1/\nu$) means $+\infty$. In other words, if $\mu =0$ or $\nu=0$, the corresponding domination conditions will vanish.

\smallskip

\noindent (iii) (Monotonicity conditions). For any $y, \bar y \in \mathbb R^n$ and almost all $\omega \in \Omega$,
\begin{equation}\label{3:M:Psi}
\langle \Psi(y) -\Psi(\bar y),\ \widehat y \rangle \leq -\mu |H\widehat y|^2.
\end{equation}
For any random variable $X$, we use the notations 
\begin{equation}\label{3:RV:Dec}
X^{(1)}:= X -\mathbb E_t[X]\quad \mbox{and}\quad X^{(2)} := \mathbb E_t[X]. 
\end{equation}
For any $X, \bar X \in L^2_{\mathcal F_T}(\Omega;\mathbb R^n)$,
\begin{equation}\label{3:M:Phi}
\mathbb E_t\left[ \left\langle \begin{pmatrix} 
\Phi\big( X, \mathbb E_t[X] \big)^{(1)} -\Phi\big( \bar X, \mathbb E_t[\bar X] \big)^{(1)}\\ 
\Phi\big( X, \mathbb E_t[X] \big)^{(2)} -\Phi\big( \bar X, \mathbb E_t[\bar X] \big)^{(2)}
\end{pmatrix},\ 
\begin{pmatrix}
\widehat X^{(1)}\\ \widehat X^{(2)}
\end{pmatrix} \right\rangle \right]
\geq \nu \mathbb E_t\left[ \left| \begin{pmatrix} 
P\widehat X^{(1)}\\ \widetilde P\widehat X^{(2)}
\end{pmatrix} \right|^2 \right].
\end{equation}
For almost all $s\in [t,T]$ and any $\Theta := (X^\top,Y^\top,Z^\top)^\top$, $\bar\Theta := (X^\top,Y^\top,Z^\top)^\top \in L^2_{\mathcal F_s}(\Omega;\mathbb R^{n(2+d)})$,
\begin{equation}\label{3:M:Gamma}
\begin{aligned}
& \mathbb E_t\left[ \left\langle \begin{pmatrix} 
\Gamma\big(s, \Theta, \mathbb E_t[\Theta] \big)^{(1)} -\Gamma\big(s, \bar \Theta, \mathbb E_t[\bar \Theta] \big)^{(1)}\\ 
\Gamma\big(s, \Theta, \mathbb E_t[\Theta] \big)^{(2)} -\Gamma\big(s, \bar \Theta, \mathbb E_t[\bar \Theta] \big)^{(2)}
\end{pmatrix},\ 
\begin{pmatrix}
\widehat \Theta^{(1)}\\ \widehat \Theta^{(2)}
\end{pmatrix} \right\rangle \right]\\
\leq\ & -\nu \mathbb E_t\left[ \left| \begin{pmatrix} 
A(s)\widehat X^{(1)}\\ \widetilde A(s)\widehat X^{(2)}
\end{pmatrix} \right|^2 \right] -\mu \mathbb E_t\left[ \left| \begin{pmatrix} 
B(s)\widehat Y^{(1)} +C(s)\widehat Z^{(1)}\\ 
\widetilde B(s)\widehat Y^{(2)} +\widetilde C(s) \widehat Z^{(2)}
\end{pmatrix} \right|^2 \right].
\end{aligned}
\end{equation}

\medskip

The domination and monotonicity conditions in Assumption (H4), especially \eqref{3:M:Phi} and \eqref{3:M:Gamma} in some form of randomization, are a bit complicated and not easy to understand. Therefore, we would like to give a remark and two examples to get some feeling. Firstly, the following remark shows a special case of Assumption (H4).

\begin{remark}
We introduce 

\medskip

\noindent{\bf Assumption (H4-S)} {\rm (Coefficients without mean-field terms)}. Let the coefficients $\Phi$ and $\Gamma$ are independent of $x'$ and $\theta'$, respectively. There exist two constants $\mu\geq 0$, $\nu\geq 0$, two matrices $H\in \mathbb R^{m_1\times n}$, $P \in \mathbb R^{m_2\times n}$, and three matrix-valued processes $A(\cdot), B(\cdot), C(\cdot) \in L^\infty(t,T;\mathbb R^{m_3\times n})$ (where $m_1, m_2, m_3 \in \mathbb N$ are given) such that the following conditions hold:

\smallskip

\noindent (i) One of the following two cases holds. {\it Case A}: $\mu>0$ and $\nu=0$. {\it Case B}: $\mu =0$ and $\nu>0$.

\smallskip

\noindent (ii) (Domination conditions). For any $x, \bar x, y, \bar y \in \mathbb R^n$, any $z, \bar z \in \mathbb R^{nd}$, and almost all $(s,\omega) \in [t,T]\times \Omega$, we have 
\begin{equation}
\left\{
\begin{aligned}
& |\Psi(y) -\Psi(\bar y)| \leq \frac 1 \mu | H\widehat y |,\\
& \big|\Phi(x) -\Phi(\bar x)\big| \leq \frac 1 \nu |P\widehat x|,\\
& \big|g(s,x,y,z) -g(s,\bar x,y,z)\big| \leq \frac 1 \nu |A(s)\widehat x|,\\
& \big| f(s,x,y,z) -f(s,x,\bar y,\bar z) \big| \leq \frac 1 \mu \big|B(s)\widehat y +C(s) \widehat z\big|.
\end{aligned}
\right.
\end{equation}

\smallskip

\noindent (iii) (Monotonicity conditions). For any $x, \bar x, y, \bar y \in \mathbb R^n$, any $\theta, \bar\theta \in \mathbb R^{n(2+d)}$, and almost all $(s,\omega) \in [t,T] \times \Omega$, we have
\begin{equation}
\left\{
\begin{aligned}
& \langle \Psi(y) -\Psi(\bar y),\ \widehat y \rangle \leq -\mu |H\widehat y|^2,\\
& \langle \Phi(x) -\Phi(\bar x),\ \widehat x \rangle \geq \nu |P\widehat x|^2,\\
& \langle \Gamma(s,\theta) -\Gamma(s, \bar\theta),\ \widehat \theta \rangle \leq -\nu |A(s) \widehat x|^2 -\mu \big|B(s)\widehat y +C(s) \widehat z\big|^2.
\end{aligned}
\right.
\end{equation}

\medskip

A straightforward verification shows that the above Assumption (H4-S) implies Assumption (H4). We notice that Assumption (H4-S) and its special cases have been extensively studied in the literature, such as \cite{HP-95, Yong-97, PW-99, Y-21+} and so on. The present paper can be regarded as an extension of these studies in the mean-field case.
\end{remark}

Now, we give an example of a decoupled linear MF-FBSDE where the coefficients satisfy Assumptions (H3) and (H4).

\begin{example}\label{3:Exa:0}
Let $\mu \geq 0$ and $\nu\geq 0$ be two constants satisfying Assumption (H4)-(i), $H\in \mathbb R^{m_1\times n}$, $P, \widetilde P \in \mathbb R^{m_2\times n}$ be three matrices, and $A(\cdot), \widetilde A(\cdot), B(\cdot), \widetilde B(\cdot) \in L^\infty(t,T;\mathbb R^{m_3\times n})$, $C(\cdot) = (C_1(\cdot),C_2(\cdot),\dots,C_d(\cdot))$, $\widetilde C(\cdot) =(\widetilde C_1(\cdot), \widetilde C_2(\cdot),\dots, \widetilde C_d(\cdot))$ with $C_i(\cdot), \widetilde C_i(\cdot) \in L^\infty(t,T;\mathbb R^{m_3\times n})$ be six matrix-valued processes. For any $y, x, x'\in \mathbb R^n$, any $\theta,\theta' \in \mathbb R^{n(2+d)}$ and any $(s,\omega) \in [t,T]\times \Omega$, we define
\begin{equation}\label{3:Coe:0}
\left\{
\begin{aligned}
& \Psi^0(y) = -\mu H^\top Hy,\\
& \Phi^0(x,x') = \nu \Big\{ P^\top P(x-x') +\widetilde P^\top \widetilde P x' \Big\},\\
& g^0(s,\theta,\theta') = -\nu \Big\{ A(s)^\top A(s) (x-x') +\widetilde A(s)^\top\widetilde A(s) x' \Big\},\\
& b^0(s,\theta,\theta') = -\mu \Big\{ B(s)^\top \Big[ B(s)(y-y') +C(s)(z-z') \Big] +\widetilde B(s)^\top \Big[ \widetilde B(s) y' +\widetilde C(s) z' \Big] \Big\},\\
& \sigma^0(s,\theta,\theta') = -\mu \Big\{ C(s)^\top \Big[ B(s)(y-y') +C(s)(z-z') \Big] +\widetilde C(s)^\top \Big[ \widetilde B(s) y' +\widetilde C(s) z' \Big] \Big\}.
\end{aligned}
\right.
\end{equation}
Similarly, we denote $\Gamma^0 :=((g^0)^\top, (b^0)^\top, (\sigma^0)^\top)^\top$. When Case A in Assumption (H4)-(i) holds, we assume that
\begin{equation}\label{3:WLOG1}
\frac{1}{\mu^2} \geq \max\left\{ \Vert H \Vert,\ 
\left\Vert \begin{pmatrix} B(\cdot) \\ \widetilde B(\cdot) \end{pmatrix} \right\Vert_{L^\infty(t,T;\mathbb R^{2m_3\times n})},\
\left\Vert \begin{pmatrix} C(\cdot)\\ \widetilde C(\cdot) \end{pmatrix} \right\Vert_{L^\infty(t,T;\mathbb R^{2m_3\times nd})}
\right\}.
\end{equation}
Then Assumptions (H4)-(ii) and (H4)-(iii) also hold for $(\Psi^0, \Phi^0, \Gamma^0)$ which are defined by \eqref{3:Coe:0}. Moreover, in this case, Assumption (H3) holds for $(\Psi^0, \Phi^0, \Gamma^0)$ with the Lipschitz constant
\begin{equation}\label{3:WLOG2}
L \geq \mu \max \Big\{ \Vert H \Vert^2,\
\left\Vert \big( B(\cdot),\ C(\cdot) \big) \right\Vert_{L^\infty(t,T;\mathbb R^{m_3\times n(1+d)})}^2
+\left\Vert \big( \widetilde B(\cdot),\ \widetilde C(\cdot) \big) \right\Vert_{L^\infty(t,T;\mathbb R^{m_3\times n(1+d)})}^2 \Big\}.
\end{equation}
When Case B in Assumption (H4)-(i) holds, we assume that
\begin{equation}\label{3:WLOG3}
\frac{1}{\nu^2} \geq \max\left\{ 
\left\Vert \begin{pmatrix} P\\ \widetilde P \end{pmatrix} \right\Vert,\
\left\Vert \begin{pmatrix} A(\cdot)\\ \widetilde A(\cdot) \end{pmatrix} \right\Vert_{L^\infty(t,T;\mathbb R^{2m_3\times n})}
\right\}.
\end{equation}
Then Assumptions (H4)-(ii) and (H4)-(iii) also hold. Moreover, in this case, Assumption (H3) holds with the Lipschitz constant
\begin{equation}\label{3:WLOG4}
L \geq \nu \max\Big\{ \Vert P \Vert^2 +\Vert \widetilde P \Vert^2,\
\left\Vert A(\cdot) \right\Vert_{L^\infty(t,T;\mathbb R^{m_3\times n})}^2 +\left\Vert \widetilde A(\cdot) \right\Vert_{L^\infty(t,T;\mathbb R^{m_3\times n})}^2
\Big\}.
\end{equation}
\end{example}

The above Example \ref{3:Exa:0} will be useful in our following analysis of this section. Another two coupled linear examples coming from the Hamiltonian systems of LQ control problems will be provided in Section \ref{Sec:LQ}. Next, we shall give a nonlinear example.

\begin{example}
Let $n=d=1$. Let $k_1\geq 1$ and $k_2\geq 1$ be two constants. For any $y,x,x'\in \mathbb R$ and any $\theta,\theta' \in \mathbb R^3$, we define
\[
\left\{
\begin{aligned}
& \Psi(y) = -k_1y +\sin y, \quad 
b(\theta,\theta') \equiv b(y,y') = -k_1y +\sin y', \quad
\sigma(\theta,\theta') \equiv \sigma(z,z') =-k_1z +\sin z',\\
& \Phi(x,x') = k_2 x +\sin x', \quad g(\theta,\theta') \equiv g(x,x') = -k_2x +\sin x'.
\end{aligned}
\right.
\]
Let $H=P=\widetilde P =1$ and 
\[
A = \widetilde A =\begin{pmatrix} 1 \\ 0\\ 0 \end{pmatrix},\qquad
B = \widetilde B =\begin{pmatrix} 0 \\ 1\\ 0 \end{pmatrix},\qquad
C = \widetilde C =\begin{pmatrix} 0 \\ 0\\ 1 \end{pmatrix}.
\]

With the same notations in Assumption (H4), we firstly try to derive the monotonicity conditions:
\[
\begin{aligned}
& \left\langle \begin{pmatrix} 
\Gamma\big(s, \Theta, \mathbb E_t[\Theta] \big)^{(1)} -\Gamma\big(s, \bar \Theta, \mathbb E_t[\bar \Theta] \big)^{(1)}\\ 
\Gamma\big(s, \Theta, \mathbb E_t[\Theta] \big)^{(2)} -\Gamma\big(s, \bar \Theta, \mathbb E_t[\bar \Theta] \big)^{(2)}
\end{pmatrix},\ 
\begin{pmatrix}
\widehat \Theta^{(1)}\\ \widehat \Theta^{(2)}
\end{pmatrix} \right\rangle \\
\leq\ & -k_2 \big| \widehat X^{(1)} \big|^2 -k_1\Big[ \big| \widehat Y^{(1)} \big|^2 +\big| \widehat Z^{(1)} \big|^2 \Big]
-(k_2-1)\big| \widehat X^{(2)} \big|^2 -(k_1-1)\Big[ \big| \widehat Y^{(2)} \big|^2 +\big| \widehat Z^{(2)} \big|^2 \Big]\\
\leq\ & -(k_2-1) \Big[ \big| \widehat X^{(1)} \big|^2 +\big| \widehat X^{(2)} \big|^2 \Big] -(k_1-1)\Big[ \widehat Y^{(1)} \big|^2 +\big| \widehat Z^{(1)} \big|^2 +\big| \widehat Y^{(2)} \big|^2 +\big| \widehat Z^{(2)} \big|^2 \Big]\\
=\ & -(k_2-1) \left| \begin{pmatrix} A\widehat X^{(1)}\\ \widetilde A \widehat X^{(2)} \end{pmatrix} \right|^2 
-(k_1-1) \left| \begin{pmatrix} B\widehat Y^{(1)} +C\widehat Z^{(1)}\\ \widetilde B\widehat Y^{(2)} +\widetilde C\widehat Z^{(2)} \end{pmatrix} \right|^2.
\end{aligned}
\]
Similarly,
\[
\big\langle \Psi(y) -\Psi(\bar y), \widehat y \big\rangle \leq -(k_1-1)|H\widehat y|^2,
\]
\[
\left\langle \begin{pmatrix} 
\Phi\big( X, \mathbb E_t[X] \big)^{(1)} -\Phi\big( \bar X, \mathbb E_t[\bar X] \big)^{(1)}\\ 
\Phi\big( X, \mathbb E_t[X] \big)^{(2)} -\Phi\big( \bar X, \mathbb E_t[\bar X] \big)^{(2)}
\end{pmatrix},\ 
\begin{pmatrix}
\widehat X^{(1)}\\ \widehat X^{(2)}
\end{pmatrix} \right\rangle \geq (k_2-1) \left| \begin{pmatrix} P\widehat X^{(1)}\\ \widetilde P\widehat X^{(2)} \end{pmatrix} \right|^2.
\]
Secondly, we try to calculate the domination conditions:
\[
\left\{
\begin{aligned}
& |\Psi(y) -\Psi(\bar y)| \leq (k_1+1)|H\widehat y|,\\
& \big|\Phi(x,x') -\Phi(\bar x, \bar x')\big| \leq 2(k_2+1) \left| \begin{pmatrix} P(\widehat x -\widehat x')\\ \widetilde P \widehat x' \end{pmatrix} \right|,\\
& \big|g(x,x') -g(\bar x,\bar x')\big| \leq 2(k_2+1) \left| \begin{pmatrix} A (\widehat x -\widehat x')\\ \widetilde A \widehat x' \end{pmatrix} \right|,\\
& \big| b(y,y') -b(\bar y,\bar y') \big| + \big| \sigma(z,z') -\sigma(\bar z,\bar z') \big| \leq 2(k_1 +1) \left| \begin{pmatrix}
B(\widehat y -\widehat y') +C(\widehat z -\widehat z')\\
\widetilde B \widehat y' +\widetilde C \widehat z'
\end{pmatrix} \right|.
\end{aligned}
\right.
\]

Based on the above calculation, we get the following results. On the one hand, when $k_1>1$ and $k_2\geq 1$, selecting $0<\mu \leq \min \{ k_1-1,\ 1/(2(k_1+1)) \}$ and $\nu=0$ leads to Assumptions (H4)-(i)-Case A, (H4)-(ii) and (H4)-(iii). On the other hand, when $k_1\geq 1$ and $k_2>1$, we choose $\mu=0$ and $0<\nu\leq \min \{ k_2-1,\ 1/(2(k_2 +1)) \}$ to get Assumptions (H4)-(i)-Case B, (H4)-(ii) and (H4)-(iii).
\end{example}

Now, we are in the position to give the main results of this section.

\begin{theorem}\label{3:THM:FBSDE}
Let Assumptions (H3) and (H4) hold for a set of coefficients $(\Psi, \Phi, \Gamma)$. Then MF-FBSDE \eqref{1:FBSDE} admits a unique solution $\theta(\cdot) \in M^2_{\mathbb F}(t,T;\mathbb R^{n(2+d)})$. Moreover, the following estimate holds:
\begin{equation}\label{3:FBSDE:Est1}
\mathbb E_t\bigg[ \sup_{s\in [t,T]}|x(s)|^2 +\sup_{s\in [t,T]}|y(s)|^2 +\int_t^T |z(s)|^2\, \mathrm ds \bigg] \leq K \mathbb E_t\big[ \mathrm I \big],
\end{equation}
where
\begin{equation}
\mathrm I = |\Psi(0)|^2 +|\Phi(0,0)|^2 +\bigg( \int_t^T |g(s,0,0)|\, \mathrm ds \bigg)^2 +\bigg( \int_t^T |b(s,0,0)|\, \mathrm ds \bigg)^2 +\int_t^T |\sigma(s,0,0)|^2\, \mathrm ds,
\end{equation}
and $K>0$ is a constant depending on $(T-t)$, the Lipschitz constant, $\mu$, $\nu$, and the bound of all $H$, $P$, $\widetilde P$, $A(\cdot)$, $\widetilde A(\cdot)$, $B(\cdot)$, $\widetilde B(\cdot)$, $C(\cdot)$, $\widetilde C(\cdot)$. Furthermore, Let $(\bar\Psi, \bar\Phi, \bar\Gamma)$ be another set of coefficients and $\bar\theta(\cdot) \in M^2_{\mathbb F}(t,T;\mathbb R^{(n+2)d})$ be a solution to the MF-FBSDE with the coefficients $(\bar\Psi, \bar\Phi, \bar\Gamma)$. We assume that $(\bar\Psi(\bar y(t)), \bar\Phi(\bar x(T),\mathbb E_t[\bar x(T)]), \bar\Gamma(\cdot, \bar\theta(\cdot),\mathbb E_t[\bar\theta(\cdot)])) \in \mathcal H[t,T]$. Then, we have
\begin{equation}\label{3:FBSDE:Est2}
\mathbb E_t\bigg[ \sup_{s\in [t,T]}|\widehat x(s)|^2 +\sup_{s\in [t,T]}|\widehat y(s)|^2 +\int_t^T |\widehat z(s)|^2\, \mathrm ds \bigg] \leq K \mathbb E_t\big[ \widehat {\mathrm I} \big],
\end{equation}
where $\widehat x(\cdot) =x(\cdot) -\bar x(\cdot)$, $\widehat y(\cdot) =y(\cdot) -\bar y(\cdot)$, $\widehat z(\cdot) =z(\cdot) -\bar z(\cdot)$,
\begin{equation}
\begin{aligned}
\widehat {\mathrm I} 
=\ & \big| \Psi(\bar y(t)) -\bar \Psi(\bar y(t)) \big|^2
+ \big| \Phi\big( \bar x(T), \mathbb E_t[\bar x(T)] \big) -\bar \Phi\big( \bar x(T), \mathbb E_t[\bar x(T)] \big) \big|^2\\
& +\bigg( \int_t^T \big| g\big( s,\bar\theta(s), \mathbb E_t[\bar\theta(s)] \big) -\bar g\big( s,\bar\theta(s), \mathbb E_t[\bar\theta(s)] \big) \big|\, \mathrm ds \bigg)^2\\
& +\bigg( \int_t^T \big| b\big( s,\bar\theta(s), \mathbb E_t[\bar\theta(s)] \big) -\bar b\big( s,\bar\theta(s), \mathbb E_t[\bar\theta(s)] \big) \big|\, \mathrm ds \bigg)^2\\
& +\int_t^T \big| \sigma \big( s,\bar\theta(s), \mathbb E_t[\bar\theta(s)] \big) -\bar \sigma \big( s,\bar\theta(s), \mathbb E_t[\bar\theta(s)] \big) \big|^2\, \mathrm ds,
\end{aligned}
\end{equation}
and $K$ is the same constant as in \eqref{3:FBSDE:Est1}.
\end{theorem}

Next, we are going to prove Theorem \ref{3:THM:FBSDE} by virtue of the method of continuation. Let us introduce the convex combination of $(\Psi, \Phi, \Gamma)$ and $(\Psi^0, \Phi^0, \Gamma^0)$ (see \eqref{3:Coe:0}):
\begin{equation}\label{3:Coe:alpha}
\big( \Psi^\alpha, \Phi^\alpha, \Gamma^\alpha \big) := \alpha \big( \Psi, \Phi, \Gamma \big) +(1-\alpha) \big( \Psi^0, \Phi^0, \Gamma^0 \big),\quad \alpha\in [0,1].
\end{equation}
For any $(\xi,\eta,\rho(\cdot)) \in \mathcal H[t,T]$ with $\rho(\cdot) =(\varphi(\cdot)^\top,\psi(\cdot)^\top,\gamma(\cdot)^\top)^\top$ and $\gamma(\cdot) = (\gamma_1(\cdot)^\top,\gamma_2(\cdot)^\top,\dots,\gamma_d(\cdot)^\top)^\top$, we continue to introduce a family of MF-FBSDEs parameterized by $\alpha\in [0,1]$ as follows:
\begin{equation}\label{3:FBSDE:alpha}
\left\{
\begin{aligned}
& \mathrm dx^\alpha(s) = \Big\{ b^\alpha\big( s,\theta^\alpha(s), \mathbb E_t[\theta^\alpha(s)] \big) +\psi(s) \Big\}\, \mathrm ds\\
& \qquad +\sum_{i=1}^d \Big\{ \sigma_i^\alpha\big( s,\theta^\alpha(s), \mathbb E_t[\theta^\alpha(s)] \big) +\gamma_i(s) \Big\}\, \mathrm dW_i(s),\quad s\in [t,T],\\
& \mathrm dy^\alpha(s) = \Big\{ g^\alpha\big( s,\theta^\alpha(s), \mathbb E_t[\theta^\alpha(s)] \big) +\varphi(s) \Big\}\, \mathrm ds +\sum_{i=1}^d z_i^\alpha(s)\, \mathrm dW_i(s), \quad s\in [t,T],\\
& x^\alpha(t) = \Psi^\alpha\big( y^\alpha(t) \big) +\xi, \qquad y^\alpha(T) = \Phi^\alpha\big( x^\alpha(T), \mathbb E_t[x^\alpha(T)]\big) +\eta.
\end{aligned}
\right.
\end{equation}
Without loss of generality, we assume that the Lipschitz constant $L$ of the original coefficients $(\Psi,\Phi,\Gamma)$ is big enough and the constant $\mu$ in Assumption (H4)-(i)-Case A and the constant $\nu$ in Assumption (H4)-(i)-Case B are small enough such that the inequalities \eqref{3:WLOG1}, \eqref{3:WLOG2}, \eqref{3:WLOG3}, \eqref{3:WLOG4} hold true. Then, a straightforward calculation shows that, for any $\alpha\in [0,1]$, the new coefficients $(\Psi^\alpha,\Phi^\alpha,\Gamma^\alpha)$ (see the definition \eqref{3:Coe:alpha}) also satisfy Assumptions (H3) and (H4) with the same Lipschitz constant $L$, $\mu$, $\nu$, $H$, $P$, $\widetilde P$, $A(\cdot)$, $\widetilde A(\cdot)$, $B(\cdot)$, $\widetilde B(\cdot)$, $C(\cdot)$, and $\widetilde C(\cdot)$ as the original coefficients $(\Psi, \Phi, \Gamma)$.

When $\alpha =0$, we substitute \eqref{3:Coe:0} into MF-FBSDE \eqref{3:FBSDE:alpha} to get
\begin{equation}\label{3:FBSDE:0}
\left\{
\begin{aligned}
& \mathrm dx^0 = \Big\{ -\mu \Big[ B^\top \big( By^{0(1)} +Cz^{0(1)} \big) +\widetilde B^\top \big( \widetilde B y^{0(2)} +\widetilde C z^{0(2)}\big) \Big] +\psi \Big\}\, \mathrm ds\\
& \quad +\sum_{i=1}^d \Big\{ -\mu \Big[ C_i^\top \big( By^{0(1)} +Cz^{0(1)} \big) +\widetilde C_i^\top \big( \widetilde B y^{0(2)} +\widetilde C z^{0(2)}\big) \Big] +\gamma_i \Big\}\, \mathrm dW_i,\quad s\in [t,T],\\
& \mathrm dy^0 = \Big\{ -\nu \Big[ A^\top A x^{0(1)} +\widetilde A^\top \widetilde A x^{0(2)} \Big] +\varphi \Big\}\, \mathrm ds +\sum_{i=1}^d z_i^0\, \mathrm dW_i,\quad s\in [t,T],\\
& x^0(t) = -\mu H^\top H y^0(t) +\xi, \qquad
y^0(T) =\nu \Big[ P^\top P x^0(T)^{(1)} +\widetilde P^\top \widetilde P x^0(T)^{(2)} \Big] +\eta,
\end{aligned}
\right.
\end{equation}
where the argument $s$ is suppressed and the decompositions (see \eqref{3:RV:Dec}) of the corresponding random variables are used for simplicity of notations. Let us discuss in two situations. (i) When Assumption (H4)-(i)-Case A holds true (i.e., $\mu>0$ and $\nu=0$), MF-FBSDE \eqref{3:FBSDE:0} is in a decoupled form. In fact, we can solve the MF-BSDE first to get $(y^0(\cdot),z^0(\cdot))$. Then we substitute $(y^0(\cdot),z^0(\cdot))$ into the MF-SDE and solve $x^0(\cdot)$. (ii) When Assumption (H4)-(i)-Case B holds true (i.e., $\mu=0$ and $\nu>0$), MF-FBSDE \eqref{3:FBSDE:0} is also in a decoupled form. The difference is that, in this case, we first solve the MF-SDE and then MF-BSDE. In summary, under Assumptions (H3) and (H4), MF-FBSDE \eqref{3:FBSDE:0} admits a unique solution $\theta^0(\cdot) \in M^2_{\mathbb F}(t,T;\mathbb R^{n(2+d)})$.

When $\alpha =1$ and $(\xi, \eta, \rho(\cdot))$ vanish, MF-FBSDE \eqref{3:FBSDE:alpha} coincides with MF-FBSDE \eqref{1:FBSDE} that we care about. Next, we will show that if for some $\alpha_0 \in [0,1)$, MF-FBSDE \eqref{3:FBSDE:alpha} is uniquely solvable for any $(\xi,\eta,\rho(\cdot)) \in \mathcal H[t,T]$, then there exists a fixed step length $\delta_0 >0$ such that for any $\alpha \in [\alpha_0,\alpha_0+\delta_0]$, the same conclusion also holds. For this aim, we firstly establish a priori estimate for the solution to MF-FBSDE \eqref{3:FBSDE:alpha}.

\begin{lemma}\label{3:Lem:a priori}
Let Assumptions (H3) and (H4) hold for a given set of coefficients $(\Psi, \Phi, \Gamma)$. Let $\alpha \in [0,1]$ and $(\xi, \eta, \rho(\cdot))$, $(\bar \xi, \bar\eta, \bar\rho(\cdot)) \in \mathcal H[t,T]$. Suppose that $\theta(\cdot)$, $\bar\theta(\cdot) \in M^2_{\mathbb F}(t,T;\mathbb R^{n(2+d)})$ satisfy MF-FBSDEs \eqref{3:FBSDE:alpha} with coefficients $(\Psi^\alpha +\xi, \Phi^\alpha +\eta, \Gamma^\alpha +\rho)$ and $(\Psi^\alpha +\bar\xi, \Phi^\alpha +\bar\eta, \Gamma^\alpha +\bar\rho)$, respectively. Then the following estimate holds:
\begin{equation}\label{3:a priori}
\mathbb E_t \bigg[ \sup_{s\in [t,T]} \big| \widehat x(s) \big|^2 +\sup_{s\in [t,T]} \big| \widehat y(s) \big|^2 +\int_t^T \big| \widehat z(s) \big|^2\, \mathrm ds \bigg] \leq K \mathbb E_t\big[ \mathrm {\widehat J} \big],
\end{equation}
where
\begin{equation}\label{3:hatJ}
\mathrm{\widehat J} = \big| \widehat \xi \big|^2 +\big| \widehat \eta \big|^2 +\bigg( \int_t^T \big| \widehat\varphi(s) \big|\, \mathrm ds \bigg)^2 +\bigg( \int_t^T \big| \widehat\psi(s) \big|\, \mathrm ds \bigg)^2 +\int_t^T \big| \widehat\gamma(s) \big|^2\, \mathrm ds,
\end{equation}
and $\widehat x(\cdot) =x(\cdot) -\bar x(\cdot)$, $\widehat\xi =\xi -\bar\xi$, etc. Here $K>0$ is a constant depending on $(T-t)$, the Lipschitz constant $L$, $\mu$, $\nu$, and the bounds of all $H$, $P$, $\widetilde P$, $A(\cdot)$, $\widetilde A(\cdot)$, $B(\cdot)$, $\widetilde B(\cdot)$, $C(\cdot)$, and $\widetilde C(\cdot)$.
\end{lemma}

\begin{proof}
By the estimate \eqref{2:SDE:Est2} for MF-SDEs, we have (the argument $s$ is suppressed and the decompositions \eqref{3:RV:Dec} are used for simplicity)
\begin{equation}\label{3:Est:Eq1}
\begin{aligned}
& \mathbb E_t \bigg[ \sup_{s\in [t,T]} \big| \widehat x(s) \big|^2 \bigg] \leq K \mathbb E_t \bigg\{ \Big| \alpha \big( \Psi(y(t)) -\Psi(\bar y(t)) \big) -(1-\alpha) \mu H^\top H \widehat y(t) +\widehat \xi \Big|^2\\
& \qquad +\bigg( \int_t^T  \Big| \alpha\Big( b\big( \bar x,\mathbb E_t[\bar x], y, \mathbb E_t[y], z, \mathbb E_t[z] \big) -b\big( \bar\theta, \mathbb E_t[\bar\theta] \big) \Big)\\
& \qquad -(1-\alpha) \mu \Big[ B^\top \big( B\widehat y^{(1)} +C\widehat z^{(1)} \big) +\widetilde B^\top \big( \widetilde B\widehat y^{(2)} +\widetilde C \widehat z^{(2)} \big) \Big] +\widehat \psi \Big|\, \mathrm ds \bigg)^2 \\
& \qquad +\int_t^T \Big| \alpha\Big( \sigma\big( \bar x,\mathbb E_t[\bar x], y, \mathbb E_t[y], z, \mathbb E_t[z] \big) -\sigma\big( \bar\theta, \mathbb E_t[\bar\theta] \big) \Big) \\
& \qquad -(1-\alpha) \mu \Big[ C^\top \big( B\widehat y^{(1)} +C\widehat z^{(1)} \big) +\widetilde C^\top \big( \widetilde B\widehat y^{(2)} +\widetilde C \widehat z^{(2)} \big) \Big] +\widehat\gamma \Big|^2\, \mathrm ds \bigg\},
\end{aligned}
\end{equation}
where the constant $K>0$ could be changed line to line. Similarly, applying the estimate \eqref{2:BSDE:Est2} for MF-BSDEs leads to
\begin{equation}\label{3:Est:Eq2}
\begin{aligned}
& \mathbb E_t \bigg[ \sup_{s\in [t,T]} \big| \widehat y(s) \big|^2 +\int_t^T \big| \widehat z \big|^2\, \mathrm ds \bigg]\\
\leq\ & K \mathbb E_t \bigg\{ \Big| \alpha \Big( \Phi\big( x(T), \mathbb E_t[x(T)] \big) -\Phi\big( \bar x(T), \mathbb E_t[\bar x(T)] \big) \Big)\\ & +(1-\alpha)\nu \Big[ P^\top P \widehat x(T)^{(1)} +\widetilde P^\top \widetilde P \widehat x(T)^{(2)} \Big] +\widehat\eta \Big|^2\\
& +\bigg( \int_t^T \Big| \alpha \Big( g\big( x,\mathbb E_t[x], \bar y, \mathbb E_t[\bar y], \bar z, \mathbb E_t[\bar z] \big) -g\big( \bar\theta,\mathbb E_t[\bar\theta] \big) \Big)\\
& -(1-\alpha) \nu \Big[ A^\top A \widehat x^{(1)} +\widetilde A^\top \widetilde A \widehat x^{(2)} \Big] +\widehat\varphi \Big|\, \mathrm ds \bigg)^2 \bigg\}.
\end{aligned}
\end{equation}

Moreover, we also apply It\^o's formula to $\langle \widehat x(\cdot),\ \widehat y(\cdot) \rangle$ to yield
\begin{equation}\label{3:Est:Eq3}
\mathrm I_1 + \mathbb E_t \Big[ \big\langle \widehat\eta,\ \widehat x(T) \big\rangle \Big] = \mathrm I_2 + \big\langle \widehat \xi,\ \widehat y(t) \big\rangle +\int_t^T \Big\{ \mathrm I_3 +\mathbb E_t \Big[ \big\langle \widehat \rho,\ \widehat \theta \big\rangle \Big] \Big\} \, \mathrm ds,
\end{equation}
where
\[
\begin{aligned}
\mathrm I_1 =\ & \alpha \mathbb E_t \Big[ \big\langle \Phi\big( x(T), \mathbb E_t[x(T)] \big) -\Phi\big( \bar x(T), \mathbb E_t[\bar x(T)] \big),\ \widehat x(T) \big\rangle \Big]\\
& +(1-\alpha) \nu \mathbb E_t \Big[ \big\langle P^\top P \widehat x(T)^{(1)} +\widetilde P^\top \widetilde P \widehat x(T)^{(2)},\ \widehat x(T) \big\rangle \Big],
\end{aligned}
\]
\[
\mathrm I_2 = \alpha \big\langle \Psi\big( y(t) \big) -\Psi\big( \bar y(t) \big),\ \widehat y(t) \big\rangle -(1-\alpha)\mu \big\langle H^\top H \widehat y(t),\ \widehat y(t) \big\rangle,
\]
and 
\[
\begin{aligned}
\mathrm I_3 =\ & \alpha \mathbb E_t \Big[ \big\langle \Gamma\big( \theta, \mathbb E_t[\theta] \big) -\Gamma\big( \bar\theta, \mathbb E_t[\bar\theta] \big),\ \widehat \theta \big\rangle \Big] -(1-\alpha) \mathbb E_t \Big[ \nu \big\langle A^\top A \widehat x^{(1)} +\widetilde A^\top \widetilde A \widehat x^{(2)},\ \widehat x \big\rangle\\
& +\mu \Big\langle B^\top \big( B\widehat y^{(1)} +C\widehat z^{(1)}\big) +\widetilde B^\top \big( \widetilde B \widehat y^{(2)} +\widetilde C \widehat z^{(2)} \big),\ \widehat y \Big\rangle \\
& +\mu \Big\langle C^\top \big( B\widehat y^{(1)} +C\widehat z^{(1)}\big) +\widetilde C^\top \big( \widetilde B \widehat y^{(2)} +\widetilde C \widehat z^{(2)} \big),\ \widehat z \Big\rangle \Big].
\end{aligned}
\]
The monotonicity conditions\eqref{3:M:Phi}, \eqref{3:M:Psi} and \eqref{3:M:Gamma} lead to 
\[
\begin{aligned}
\mathrm I_1 =\ & \alpha \mathbb E_t \left[ \left\langle 
\begin{pmatrix}
\Phi\big( x(T), \mathbb E_t[x(T)] \big)^{(1)} -\Phi\big( \bar x(T), \mathbb E_t[\bar x(T)] \big)^{(1)}\\
\Phi\big( x(T), \mathbb E_t[x(T)] \big)^{(2)} -\Phi\big( \bar x(T), \mathbb E_t[\bar x(T)] \big)^{(2)}
\end{pmatrix},\ \begin{pmatrix}
\widehat x(T)^{(1)}\\ \widehat x(T)^{(2)}
\end{pmatrix}
\right\rangle \right]\\
& +(1-\alpha) \nu \mathbb E_t \left[ \left| \begin{pmatrix}
P\widehat x(T)^{(1)}\\ \widetilde P \widehat x(T)^{(2)}
\end{pmatrix} \right|^2 \right] \geq \nu \mathbb E_t \left[ \left| \begin{pmatrix}
P\widehat x(T)^{(1)}\\ \widetilde P \widehat x(T)^{(2)}
\end{pmatrix} \right|^2 \right],
\end{aligned}
\]
\[
\mathrm I_2 \leq -\mu \big| H\widehat y(t) \big|^2
\]
and 
\[
\begin{aligned}
\mathrm I_3 =\ & \alpha \mathbb E_t \left[ \left\langle 
\begin{pmatrix}
\Gamma\big( \theta, \mathbb E_t[\theta] \big)^{(1)} -\Gamma\big( \bar \theta, \mathbb E_t[\bar \theta] \big)^{(1)}\\
\Gamma\big( \theta, \mathbb E_t[\theta] \big)^{(2)} -\Gamma\big( \bar \theta, \mathbb E_t[\bar \theta] \big)^{(2)}
\end{pmatrix},\ \begin{pmatrix}
\widehat \theta^{(1)}\\ \widehat \theta^{(2)}
\end{pmatrix}
\right\rangle \right]\\
& -(1-\alpha) \nu \mathbb E_t \left[ \left| \begin{pmatrix}
A\widehat x^{(1)}\\ \widetilde A \widehat x^{(2)}
\end{pmatrix} \right|^2 \right] -(1-\alpha)\mu \mathbb E_t \left[ \left| \begin{pmatrix}
B\widehat y^{(1)} +C\widehat z^{(1)}\\ 
\widetilde B\widehat y^{(2)} +\widetilde C\widehat z^{(2)}
\end{pmatrix} \right|^2 \right]\\
\leq\ & - \nu \mathbb E_t \left[ \left| \begin{pmatrix}
A\widehat x^{(1)}\\ \widetilde A \widehat x^{(2)}
\end{pmatrix} \right|^2 \right] -\mu \mathbb E_t \left[ \left| \begin{pmatrix}
B\widehat y^{(1)} +C\widehat z^{(1)}\\ 
\widetilde B\widehat y^{(2)} +\widetilde C\widehat z^{(2)}
\end{pmatrix} \right|^2 \right].
\end{aligned}
\]
Therefore, \eqref{3:Est:Eq3} is reduced to
\begin{equation}\label{3:Est:Eq4}
\begin{aligned}
& \mathbb E_t \left\{ \mu \big| H\widehat y(t) \big|^2
+\nu \left| \begin{pmatrix}
P\widehat x(T)^{(1)}\\ \widetilde P \widehat x(T)^{(2)}
\end{pmatrix} \right|^2
+\int_t^T \left[
\nu \left| \begin{pmatrix}
A\widehat x^{(1)}\\ \widetilde A \widehat x^{(2)}
\end{pmatrix} \right|^2
+\mu \left| \begin{pmatrix}
B\widehat y^{(1)} +C\widehat z^{(1)}\\ 
\widetilde B\widehat y^{(2)} +\widetilde C\widehat z^{(2)}
\end{pmatrix} \right|^2
\right]\, \mathrm ds \right\}\\
& \leq \mathbb E_t \bigg\{ \big\langle \widehat \xi,\ \widehat y(t) \big\rangle
-\big\langle \widehat\eta,\ \widehat x(T) \big\rangle
+\int_t^T \big\langle \widehat \rho,\ \widehat \theta \big\rangle \, \mathrm ds \bigg\}.
\end{aligned}
\end{equation}

We will divide the remaining proof into two cases according to Assumption (H4)-(i).

{\bf Case A: $\mu>0$ and $\nu=0$.} By applying the domination conditions \eqref{3:D} to the estimate \eqref{3:Est:Eq1}, we have
\begin{equation}\label{3:Est:Eq5}
\begin{aligned}
\mathbb E_t \Bigg[ \sup_{s\in [t,T]} \big| \widehat x(s) \big|^2 \bigg] \leq\ & K \mathbb E_t \bigg\{ \big| H\widehat y(t) \big|^2 +\int_t^T \left|
\begin{pmatrix} 
B\widehat y^{(1)} +C\widehat z^{(1)} \\
\widetilde B\widehat y^{(2)} +\widetilde C\widehat z^{(2)}
\end{pmatrix}
\right|^2\, \mathrm ds\\
& +\big|\widehat \xi\big|^2 +\bigg( \int_t^T \big| \widehat \psi \big|\, \mathrm ds \bigg)^2 +\int_t^T \big|\widehat \gamma\big|^2\, \mathrm ds \Bigg\}.
\end{aligned}
\end{equation}
Applying the Lipschitz condition to \eqref{3:Est:Eq2} leads to
\begin{equation}\label{3:Est:Eq6}
\mathbb E_t \bigg[ \sup_{s\in [t,T]} \big| \widehat y(s) \big|^2 +\int_t^T \big| \widehat z \big|^2\, \mathrm ds \bigg] \leq K \mathbb E_t \bigg\{ \sup_{s\in [t,T]} \big|\widehat x(s)\big|^2 +\big|\widehat \eta\big|^2 +\bigg( \int_t^T \big| \widehat\varphi \big|\, \mathrm ds \bigg)^2 \bigg\}.
\end{equation}
Combining \eqref{3:Est:Eq5} and \eqref{3:Est:Eq6} yields
\begin{equation}\label{3:Est:Eq7}
\begin{aligned}
& \mathbb E_t \Bigg[ \sup_{s\in [t,T]} \big| \widehat x(s) \big|^2 +\sup_{s\in [t,T]} \big| \widehat y(s) \big|^2 +\int_t^T \big| \widehat z \big|^2\, \mathrm ds \bigg]\\ 
\leq\ & K \mathbb E_t \bigg\{ \big| H\widehat y(t) \big|^2 +\int_t^T \left|
\begin{pmatrix} 
B\widehat y^{(1)} +C\widehat z^{(1)} \\
\widetilde B\widehat y^{(2)} +\widetilde C\widehat z^{(2)}
\end{pmatrix}
\right|^2\, \mathrm ds +\mathrm{\widehat J} \bigg\},
\end{aligned}
\end{equation}
where $\mathrm{\widehat J}$ is defined by \eqref{3:hatJ}. We continue to combine \eqref{3:Est:Eq7} and \eqref{3:Est:Eq4} to have
\begin{equation}\label{3:Est:Eq8}
\begin{aligned}
& \mathbb E_t \Bigg[ \sup_{s\in [t,T]} \big| \widehat x(s) \big|^2 +\sup_{s\in [t,T]} \big| \widehat y(s) \big|^2 +\int_t^T \big| \widehat z \big|^2\, \mathrm ds \bigg]\\
\leq\ & K_1 \mathbb E_t \bigg\{ \mathrm{\widehat J} +\big\langle \widehat \xi,\ \widehat y(t) \big\rangle
-\big\langle \widehat\eta,\ \widehat x(T) \big\rangle
+\int_t^T \big\langle \widehat \rho,\ \widehat \theta \big\rangle \, \mathrm ds \bigg\}\\
\leq\ & \mathbb E_t \bigg\{ K_1 \mathrm {\widehat J} +K_1^2 \bigg[ \big| \widehat\xi \big|^2 +\big| \widehat\eta \big|^2 +\bigg( \int_t^T \big|\widehat \varphi \big|\, \mathrm ds \bigg)^2 +\bigg( \int_t^T \big|\widehat \psi \big|\, \mathrm ds \bigg)^2 \bigg]\\
& +\frac{K_1^2}{2} \int_t^T \big|\widehat \gamma\big|^2\, \mathrm ds +\frac 1 2 \bigg[ \sup_{s\in [t,T]} \big| \widehat x(s) \big|^2 +\sup_{s\in [t,T]} \big| \widehat y(s) \big|^2 +\int_t^T \big| \widehat z \big|^2\, \mathrm ds \bigg] \bigg\},
\end{aligned}
\end{equation}
where the inequality $ab \leq \frac{1}{4\varepsilon} a^2 +\varepsilon b^2$ for any $\varepsilon >0$ and any $a,b\in\mathbb R$ was used. Clearly, the above \eqref{3:Est:Eq8} implies the desired estimate \eqref{3:a priori}. We finish the proof in this case.

{\bf Case B: $\mu=0$ and $\nu>0$.} In this case, we apply the domination conditions \eqref{3:D} to the estimate \eqref{3:Est:Eq2} to get
\begin{equation}\label{3:Est:Eq9}
\begin{aligned}
& \mathbb E_t \bigg[ \sup_{s\in [t,T]} \big| \widehat y(s) \big|^2 +\int_t^T \big| \widehat z \big|^2\, \mathrm ds \bigg]\\
\leq\ & K \mathbb E_t \bigg\{ \left| \begin{pmatrix}
P\widehat x(T)^{(1)}\\ \widetilde P \widehat x(T)^{(2)}
\end{pmatrix} \right|^2
+\int_t^T \left| \begin{pmatrix}
A\widehat x^{(1)}\\ \widetilde A \widehat x^{(2)}
\end{pmatrix} \right|^2\, \mathrm ds +\big|\widehat \eta\big|^2 +\bigg( \int_t^T \big|\widehat\varphi\big|\, \mathrm ds \bigg)^2 \bigg\}.
\end{aligned}
\end{equation}
By the Lipschitz conditions on the coefficients $\Psi$, $b$ and $\sigma$, we deduce from \eqref{3:Est:Eq1} to have
\begin{equation}\label{3:Est:Eq10}
\begin{aligned}
\mathbb E_t \bigg[ \sup_{s\in [t,T]} \big| \widehat x(s) \big|^2 \bigg] \leq\ & K \mathbb E_t \bigg\{ \sup_{s\in [t,T]} \big| \widehat y(s) \big|^2 +\int_t^T \big| \widehat z \big|^2\, \mathrm ds\\
& + \big| \widehat \xi \big|^2 +\bigg( \int_t^T \big|\widehat\psi\big|\, \mathrm ds \bigg)^2 +\int_t^T \big| \widehat\gamma \big|^2\, \mathrm ds \bigg\}.
\end{aligned}
\end{equation}
Then, \eqref{3:Est:Eq10} and \eqref{3:Est:Eq9} work together to yield
\begin{equation}\label{3:Est:Eq11}
\begin{aligned}
& \mathbb E_t \Bigg[ \sup_{s\in [t,T]} \big| \widehat x(s) \big|^2 +\sup_{s\in [t,T]} \big| \widehat y(s) \big|^2 +\int_t^T \big| \widehat z \big|^2\, \mathrm ds \bigg]\\ 
\leq\ & K \mathbb E_t \bigg\{ \left| \begin{pmatrix}
P\widehat x(T)^{(1)}\\ \widetilde P \widehat x(T)^{(2)}
\end{pmatrix} \right|^2
+\int_t^T \left| \begin{pmatrix}
A\widehat x^{(1)}\\ \widetilde A \widehat x^{(2)}
\end{pmatrix} \right|^2\, \mathrm ds +\mathrm{\widehat J} \bigg\},
\end{aligned}
\end{equation}
where $\mathrm{\widehat J}$ is defined by \eqref{3:hatJ}. Then, combining \eqref{3:Est:Eq11} and \eqref{3:Est:Eq4} leads to
\begin{equation}
\begin{aligned}
& \mathbb E_t \Bigg[ \sup_{s\in [t,T]} \big| \widehat x(s) \big|^2 +\sup_{s\in [t,T]} \big| \widehat y(s) \big|^2 +\int_t^T \big| \widehat z \big|^2\, \mathrm ds \bigg]\\
\leq\ & K_2 \mathbb E_t \bigg\{ \mathrm{\widehat J} +\big\langle \widehat \xi,\ \widehat y(t) \big\rangle
-\big\langle \widehat\eta,\ \widehat x(T) \big\rangle
+\int_t^T \big\langle \widehat \rho,\ \widehat \theta \big\rangle \, \mathrm ds \bigg\}.
\end{aligned}
\end{equation}
By an argument similar to \eqref{3:Est:Eq8}, we know that the above inequality implies the desired priori estimate \eqref{3:a priori}. The whole proof of the lemma is completed. 
\end{proof}

With the help of Lemma \ref{3:Lem:a priori}, we provide the following continuation lemma.

\begin{lemma}\label{3:Lem:Con}
Let Assumptions (H3) and (H4) hold. If for some $\alpha_0\in [0,1)$, MF-FBSDE \eqref{3:FBSDE:alpha} is uniquely solvable in $M^2_{\mathbb F}(t,T;\mathbb R^{n(2+d)})$ for any given $(\xi, \eta, \rho(\cdot)) \in \mathcal H[t,T]$, then there exists an absolute constant $\delta_0>0$ such that the same conclusion also holds for $\alpha =\alpha_0+\delta$ where $\delta \in (0,\delta_0]$ and $\alpha \leq 1$.
\end{lemma}

\begin{proof}
Let $\delta_0>0$ be determined below, and $\delta\in (0,\delta_0]$. For any $\theta(\cdot) \in M^2_{\mathbb F}(t,T;\mathbb R^{n(2+d)})$, we consider the following MF-FBSDE (compared to \eqref{3:FBSDE:alpha} with $\alpha =\alpha_0 +\delta$):
\begin{equation}\label{3:FBSDE:Con}
\left\{
\begin{aligned}
& \mathrm dX(s) =\Big\{ b^{\alpha_0}\big( s, \Theta(s), \mathbb E_t[\Theta(s)] \big) +\widetilde \psi(s) \Big\}\, \mathrm ds\\
& \qquad +\sum_{i=1}^d \Big\{ \sigma_i^{\alpha_0}\big( s, \Theta(s), \mathbb E_t[\Theta(s)] \big) +\widetilde \gamma_i(s) \Big\}\, \mathrm dW_i(s),\quad s\in [t,T],\\
& \mathrm dY(s) =\Big\{ g^{\alpha_0}\big( s, \Theta(s), \mathbb E_t[\Theta(s)] \big) +\widetilde \varphi(s) \Big\}\, \mathrm ds +\sum_{i=1}^d Z_i(s)\, \mathrm dW_i(s),\quad s\in [t,T],\\
& X(t) =\Psi^{\alpha_0}\big( Y(t) \big) +\widetilde \xi,\qquad Y(T) = \Phi^{\alpha_0} \big( X(T), \mathbb E_t[X(T)] \big) +\widetilde\eta,
\end{aligned}
\right.
\end{equation}
where
\begin{equation}
\left\{
\begin{aligned}
& \widetilde \xi := \delta \big[ \Psi(y(t)) -\Psi^0(y(t)) \big] +\xi,\\
& \widetilde \eta := \delta \big[ \Phi\big( x(T), \mathbb E_t[x(T)] \big) -\Phi^0\big( x(T), \mathbb E_t[x(T)] \big) \big] +\eta,\\
& \widetilde \rho(\cdot) := \delta \big[ \Gamma\big( \cdot, \theta(\cdot), \mathbb E_t[\theta(\cdot)] \big) -\Gamma^0\big( \cdot, \theta(\cdot), \mathbb E_t[\theta(\cdot)] \big] +\rho(\cdot),
\end{aligned}
\right.
\end{equation}
and $(\Psi^0, \Phi^0, \Gamma^0)$ is given by \eqref{3:Coe:0}. It is easy to verify that $(\widetilde \xi, \widetilde\eta, \widetilde \rho(\cdot)) \in \mathcal H[t,T]$. Then, by our assumption, MF-FBSDE \eqref{3:FBSDE:Con} admits a unique solution $\Theta(\cdot) \in M^2_{\mathbb F}(t,T;\mathbb R^{n(2+d)})$. Noting the arbitrariness of $\theta(\cdot)$, we have established a mapping
\[
\Theta(\cdot) = \mathscr T_{\alpha_0+\delta} \big( \theta(\cdot) \big) :
M^2_{\mathbb F}(t,T;\mathbb R^{n(2+d)}) \rightarrow M^2_{\mathbb F}(t,T;\mathbb R^{n(2+d)}).
\]
In the following, we shall prove that this mapping is a contraction when $\delta$ is small.

Let $\theta(\cdot)$, $\bar\theta(\cdot) \in M^2_{\mathbb F}(t,T;\mathbb R^{n(2+d)})$ and denote $\Theta(\cdot) = \mathscr T_{\alpha_0+\delta}(\theta(\cdot))$, $\bar \Theta(\cdot) = \mathscr T_{\alpha_0+\delta}(\bar \theta(\cdot))$. Moreover, denote $\widehat\theta(\cdot) =\theta(\cdot) -\bar \theta(\cdot)$, $\widehat\Theta(\cdot) =\Theta(\cdot) -\bar \Theta(\cdot)$, etc. By applying Lemma \ref{3:Lem:a priori}, we have
\[
\begin{aligned}
& \big\Vert \widehat\Theta(\cdot) \big\Vert_{M^2_{\mathbb F}(t,T;\mathbb R^{n(2+d)})}^2 = \mathbb E\bigg[ \sup_{s\in [t,T]} \big| \widehat X(s) \big|^2 +\sup_{s\in [t,T]} \big| \widehat Y(s) \big|^2 +\int_t^T \big| \widehat Z(s) \big|^2\, \mathrm ds \bigg]\\
\leq\ & K\delta^2 \mathbb E\bigg\{ \Big| \Big[ \Psi\big( y(t) \big) -\Psi\big( \bar y(t) \big) \Big] -\Big[ \Psi^0\big( y(t) \big) -\Psi^0\big( \bar y(t) \big) \Big] \Big|^2\\
& +\Big| \Big[ \Phi \big( x(T), \mathbb E_t[x(T)] \big) -\Phi \big( \bar x(T), \mathbb E_t[\bar x(T)] \big) \Big] -\Big[ \Phi^0 \big( x(T), \mathbb E_t[x(T)] \big) -\Phi^0 \big( \bar x(T), \mathbb E_t[\bar x(T)] \big) \Big] \Big|^2\\
& +\bigg( \int_t^T \Big| \Big[ g\big(\theta, \mathbb E_t[\theta]\big) -g\big(\bar\theta, \mathbb E_t[\bar\theta]\big) \Big] -\Big[ g^0\big(\theta, \mathbb E_t[\theta]\big) -g^0\big(\bar\theta, \mathbb E_t[\bar\theta]\big) \Big] \Big|\, \mathrm ds \bigg)^2\\
& +\bigg( \int_t^T \Big| \Big[ b\big(\theta, \mathbb E_t[\theta]\big) -b\big(\bar\theta, \mathbb E_t[\bar\theta]\big) \Big] -\Big[ b^0\big(\theta, \mathbb E_t[\theta]\big) -b^0\big(\bar\theta, \mathbb E_t[\bar\theta]\big) \Big] \Big|\, \mathrm ds \bigg)^2\\
& +\int_t^T \Big| \Big[ \sigma\big(\theta, \mathbb E_t[\theta]\big) -\sigma\big(\bar\theta, \mathbb E_t[\bar\theta]\big) \Big] -\Big[ \sigma^0\big(\theta, \mathbb E_t[\theta]\big) -\sigma^0\big(\bar\theta, \mathbb E_t[\bar\theta]\big) \Big] \Big|^2\, \mathrm ds \bigg\}.
\end{aligned}
\]
Then, due to the Lipschitz continuity of the coefficients $(\Psi,\Phi,\Gamma)$ and $(\Psi^0, \Phi^0, \Gamma^0)$, there exists a constant $K_3>0$ independent of $\alpha_0$ and $\delta$ such that
\[
\big\Vert \widehat\Theta(\cdot) \big\Vert_{M^2_{\mathbb F}(t,T;\mathbb R^{n(2+d)})}^2 \leq K_3\delta^2 \big\Vert \widehat\theta(\cdot) \big\Vert_{M^2_{\mathbb F}(t,T;\mathbb R^{n(2+d)})}^2.
\]
Choose $\delta_0 = 1/(2\sqrt{K_3})$. Then, for any $\delta\in (0,\delta_0]$, the above inequality implies that $\mathscr T_{\alpha_0+\delta}$ is a contraction mapping. Consequently, it admits a unique fixed point which is just the unique solution to MF-FBSDE \eqref{3:FBSDE:alpha} with $\alpha =\alpha_0+\delta$. The proof is completed.
\end{proof}

\begin{proof}[Proof of Theorem \ref{3:THM:FBSDE}]
Firstly, the unique solvability of MF-FBSDE \eqref{1:FBSDE} in the space $M^2_{\mathbb F}(t,T;\mathbb R^{n(2+d)})$ is obtained by virtue of the unique solvability of MF-FBSDE \eqref{3:FBSDE:0} and Lemma \ref{3:Lem:Con}. Secondly, by letting $\alpha=1$, $(\xi, \eta, \rho(\cdot)) = (0,0,0)$ and
\[
\begin{aligned}
\big(\bar\xi, \bar\eta, \bar\rho(\cdot)\big) = \Big( & \bar\Psi\big(\bar y(t)\big) -\Psi\big(\bar y(t)\big),\ \bar\Phi\big( \bar x(T), \mathbb E_t[\bar x(T)] \big) -\Phi\big( \bar x(T), \mathbb E_t[\bar x(T)] \big),\\
& \bar\Gamma\big(\cdot, \bar\theta(\cdot), \mathbb E_t[\bar\theta(\cdot)]\big) -\Gamma\big(\cdot, \bar\theta(\cdot), \mathbb E_t[\bar\theta(\cdot)]\big) \Big),
\end{aligned}
\]
we obtain \eqref{3:FBSDE:Est2} from \eqref{3:a priori}. Finally, by selecting $(\bar\Psi, \bar\Phi, \bar\Gamma) =(0,0,0)$, we get \eqref{3:FBSDE:Est1} from \eqref{3:FBSDE:Est2} and complete the proof.
\end{proof}

Now, we give a remark to end this section.

\begin{remark}\label{3:Rem:Sym}
There exists a symmetrical version of the monotonicity conditions in Assumption (H4)-(iii) as follows:

\medskip

\noindent{\bf Assumption (H4)-(iii)$'$.} {\rm For any $y, \bar y \in \mathbb R^n$ and almost all $\omega \in \Omega$,
\begin{equation}
\langle \Psi(y) -\Psi(\bar y),\ \widehat y \rangle \geq \mu |H\widehat y|^2.
\end{equation}
For any $X, \bar X \in L^2_{\mathcal F_T}(\Omega;\mathbb R^n)$,
\begin{equation}
\mathbb E_t\left[ \left\langle \begin{pmatrix} 
\Phi\big( X, \mathbb E_t[X] \big)^{(1)} -\Phi\big( \bar X, \mathbb E_t[\bar X] \big)^{(1)}\\ 
\Phi\big( X, \mathbb E_t[X] \big)^{(2)} -\Phi\big( \bar X, \mathbb E_t[\bar X] \big)^{(2)}
\end{pmatrix},\ 
\begin{pmatrix}
\widehat X^{(1)}\\ \widehat X^{(2)}
\end{pmatrix} \right\rangle \right]
\leq -\nu \mathbb E_t\left[ \left| \begin{pmatrix} 
P\widehat X^{(1)}\\ \widetilde P\widehat X^{(2)}
\end{pmatrix} \right|^2 \right].
\end{equation}
For almost all $s\in [t,T]$ and any $\Theta := (X^\top,Y^\top,Z^\top)^\top$, $\bar\Theta := (X^\top,Y^\top,Z^\top)^\top \in L^2_{\mathcal F_s}(\Omega;\mathbb R^{n(2+d)})$,
\begin{equation}
\begin{aligned}
& \mathbb E_t\left[ \left\langle \begin{pmatrix} 
\Gamma\big(s, \Theta, \mathbb E_t[\Theta] \big)^{(1)} -\Gamma\big(s, \bar \Theta, \mathbb E_t[\bar \Theta] \big)^{(1)}\\ 
\Gamma\big(s, \Theta, \mathbb E_t[\Theta] \big)^{(2)} -\Gamma\big(s, \bar \Theta, \mathbb E_t[\bar \Theta] \big)^{(2)}
\end{pmatrix},\ 
\begin{pmatrix}
\widehat \Theta^{(1)}\\ \widehat \Theta^{(2)}
\end{pmatrix} \right\rangle \right]\\
\geq\ & \nu \mathbb E_t\left[ \left| \begin{pmatrix} 
A(s)\widehat X^{(1)}\\ \widetilde A(s)\widehat X^{(2)}
\end{pmatrix} \right|^2 \right] +\mu \mathbb E_t\left[ \left| \begin{pmatrix} 
B(s)\widehat Y^{(1)} +C(s)\widehat Z^{(1)}\\ 
\widetilde B(s)\widehat Y^{(2)} +\widetilde C(s) \widehat Z^{(2)}
\end{pmatrix} \right|^2 \right].
\end{aligned}
\end{equation}
}

\medskip

\noindent In fact, it is easy to verify that, if $\theta(\cdot) = (x(\cdot)^\top, y(\cdot)^\top, z(\cdot)^\top)^\top \in M^2_{\mathbb F}(t,T;\mathbb R^{n(2+d)})$ is a solution to MF-FBSDE \eqref{1:FBSDE} with the coefficients $(\Psi, \Phi, \Gamma)$, then
\begin{equation}\label{3:Sym:Trans}
\widetilde \theta(\cdot) = (\widetilde x(\cdot)^\top, \widetilde y(\cdot)^\top, \widetilde z(\cdot)^\top)^\top := \big( x(\cdot)^\top, -y(\cdot)^\top, -z(\cdot)^\top \big)^\top
\end{equation}
is a solution to another MF-FBSDE with the coefficients
\[
\begin{aligned}
& \widetilde\Psi(\widetilde y) := \Psi(-\widetilde y),\qquad \widetilde\Phi(\widetilde x, \widetilde x') := -\Phi(\widetilde x, \widetilde x'),\\
& \widetilde \Gamma(s, \widetilde \theta, \widetilde\theta') = \big(   -g\big(s,\widetilde x, -\widetilde y, -\widetilde z, \widetilde x', -\widetilde y', -\widetilde z' \big)^\top, b\big(s,\widetilde x, -\widetilde y, -\widetilde z, \widetilde x', -\widetilde y', -\widetilde z' \big)^\top,\\
& \hskip 23.5mm \sigma \big(s,\widetilde x, -\widetilde y, -\widetilde z, \widetilde x', -\widetilde y', -\widetilde z' \big)^\top \big)^\top.
\end{aligned}
\]
We continue to verify that, if the coefficients $(\Psi,\Phi,\Gamma)$ satisfy Assumption (H4)-(iii) (resp. Assumption (H4)-(iii)$'$), then the coefficients $(\widetilde \Psi, \widetilde \Phi, \widetilde \Gamma)$ will satisfy Assumption (H4)-(iii)$'$ (resp. Assumption (H4)-(iii)). By virtue of the invertible transformation \eqref{3:Sym:Trans}, all conclusions in this section are also valid when Assumption (H4)-(iii) is replaced by Assumption (H4)-(iii)$'$.
\end{remark}

\section{Application to MF-LQ problems}\label{Sec:LQ}

In this section, we shall consider some LQ optimal control problems driven by an MF-SDE and an MF-BSDE respectively. The Hamiltonian systems arising from these LQ problems will be found to be MF-FBSDEs with domination-monotonicity conditions, then they are uniquely solvable by Theorem \ref{3:THM:FBSDE}. In fact, to study the (unique) solvability of Hamiltonian systems is one of our research motivations.

\subsection{Forward MF-LQ control problem}

In the first LQ problem, we consider the following linear controlled MF-SDE:
\begin{equation}\label{4.1:Sys}
\left\{
\begin{aligned}
& \mathrm dx(s) =\Big\{ A(s)x(s) +\bar A(s) \mathbb E_t[x(s)] +B(s)u(s) +\bar B(s)\mathbb E_t[u(s)] +\alpha(s) \Big\}\, \mathrm ds\\
& \quad +\sum_{i=1}^d \Big\{ C_i(s)x(s) +\bar C_i(s) \mathbb E_t[x(s)] +D_i(s)u(s) +\bar D_i(s)\mathbb E_t[u(s)] +\beta_i(s) \Big\}\, \mathrm dW_i(s),\\ 
& \hskip 12cm s\in [t,T],\\
& x(t) = H\xi +x_t,
\end{aligned}
\right.
\end{equation}
where $A(\cdot), \bar A(\cdot), C_i(\cdot), \bar C_i(\cdot) \in L^\infty(t,T;\mathbb R^{n\times n})$, $B(\cdot), \bar B(\cdot), D_i(\cdot), \bar D_i(\cdot) \in L^\infty(t,T;\mathbb R^{n\times m})$, $H\in \mathbb R^{n\times n}$, $\alpha(\cdot) \in L^2_{\mathbb F}(\Omega;L([t,T];\mathbb R^n))$, $\beta_i(\cdot) \in L^2_{\mathbb F}(t,T;\mathbb R^n)$, and $x_t\in L^2_{\mathcal F_t}(\Omega;\mathbb R^n)$ ($i=1,2,\dots,d$). The pair $(\xi, u(\cdot)) \in L^2_{\mathcal F_t}(\Omega;\mathbb R^n) \times L^2_{\mathbb F}(t,T;\mathbb R^m)$ is called an {\it admissible control}. By Proposition \ref{2:Prop:SDE}, for any admissible control $(\xi, u(\cdot))\in L^2_{\mathcal F_t}(\Omega;\mathbb R^n) \times L^2_{\mathbb F}(t,T;\mathbb R^m)$, MF-SDE \eqref{4.1:Sys} admits a unique solution $x(\cdot)\equiv x(\cdot;\xi,u(\cdot)) \in L^2_{\mathbb F}(\Omega;C([t,T];\mathbb R^n))$ which is called the {\it admissible state process} under $(\xi,u(\cdot))$. Moreover, $(x(\cdot),\xi,u(\cdot))$ is called an {\it admissible triple}. For convenience, we denote $C(\cdot) := (C_1(\cdot)^\top, C_2(\cdot)^\top, \dots, C_d(\cdot)^\top)^\top$ and $D(\cdot) := (D_1(\cdot)^\top, D_2(\cdot)^\top, \dots, D_d(\cdot)^\top)^\top$.

We notice that, in most of literature on LQ optimal control problems (see Yong and Zhou \cite{YZh-99} for the problems driven by SDEs and  \cite{Yong-13, NLZh-16, S-17, W-19, LLY-20} for the problems driven by MF-SDEs), the initial condition is fixed to be $x(t) =x_t$. Clearly, this is a special case of our research, i.e. $H=0$. Here, we consider a general situation, i.e., we can select $\xi\in L^2_{\mathcal F_t}(\Omega;\mathbb R^n)$ to change the initial value $x(t)$. The introduction of matrix $H$ enables our formulation to be better meet the various actual applications. For example, if 
\[
H = \begin{pmatrix} 
1 & 0 & \cdots & 0 \\
0 & 0 & \cdots & 0 \\
\vdots & \vdots & \ddots & \vdots \\
0 & 0 & \cdots & 0
\end{pmatrix},
\]
then we can only change the value of the first component of $x(t)$.

Besides the controlled system \eqref{4.1:Sys}, we are also given an objective functional in a quadratic form:
\begin{equation}
\begin{aligned}
J\big( \xi, u(\cdot) \big) =\ & \frac 1 2 \mathbb E_t \bigg\{ \langle M \xi,\ \xi \rangle +\langle Gx(T),\ x(T) \rangle +\big\langle \bar G\mathbb E_t[x(T)],\ \mathbb E_t[x(T)] \big\rangle \\
& +\int_t^T \Big[ \langle Q(s)x(s),\ x(s) \rangle +\big\langle \bar Q(s)\mathbb E_t[x(s)],\ \mathbb E_t[x(s)] \big\rangle\\
& +\langle R(s)u(s),\ u(s) \rangle +\big\langle \bar R(s)\mathbb E_t[u(s)],\ \mathbb E_t[u(s)] \big\rangle \Big]\, \mathrm ds \bigg\},
\end{aligned}
\end{equation}
where $M, G, \bar G \in\mathbb S^n$, $Q(\cdot), \bar Q(\cdot) \in L^\infty(t,T;\mathbb S^n)$ and $R(\cdot), \bar R(\cdot) \in L^\infty(t,T;\mathbb S^m)$. Clearly, for any $(\xi, u(\cdot)) \in L^2_{\mathcal F_t}(\Omega;\mathbb R^n) \times L^2_{\mathbb F}(t,T;\mathbb R^m)$, $J(\xi,u(\cdot))$ is well-defined.

We propose a mean-field type forward LQ (MF-FLQ, for short) control problem as follows:

\medskip

\noindent{\bf Problem (MF-FLQ).} Find an admissible control $(\xi^*,u^*(\cdot)) \in L^2_{\mathcal F_t}(\Omega;\mathbb R^n) \times L^2_{\mathbb F}(t,T;\mathbb R^m)$ such that
\begin{equation}\label{4.1:Opt}
J\big( \xi^*, u^*(\cdot) \big) =\essinf_{(\xi, u(\cdot)) \in L^2_{\mathcal F_t}(\Omega;\mathbb R^n) \times L^2_{\mathbb F}(t,T;\mathbb R^m)} J\big( \xi, u(\cdot) \big).
\end{equation}
$(\xi^*,u^*(\cdot))$ satisfying \eqref{4.1:Opt} is called an {\it optimal control}, $x^*(\cdot) \equiv x(\cdot;\xi^*,u^*(\cdot))$ is called the corresponding {\it optimal state process}, and $(x^*(\cdot),\xi^*,u^*(\cdot))$ is called an {\it optimal triple}.

\medskip

As usual, for a matrix $M\in \mathbb S^n$, when $M$ is positive semi-definite (resp. positive definite, negative semi-definite, negative definite), we denote $M\geq 0$ (resp. $>0$, $\leq 0$, $<0$). Moreover, for a mapping $M: [t,T] \rightarrow \mathbb S^n$, we denote $M(\cdot) \geq 0$ (resp. $>0$, $\leq 0$, $<0$) when $M(s) \geq 0$ (resp. $>0$, $\leq 0$, $<0$) for almost all $s\in [t,T]$. When there exists a constant $\delta>0$ such that $M(\cdot) -\delta I_n \geq 0$ (resp. $M(\cdot) +\delta I_n \leq 0$), we denote $M(\cdot) \gg 0$ (resp. $M(\cdot)\ll 0$). Now, we introduce the following assumption called the uniformly positive definiteness condition (PD, for short) for the weighting matrices in the objective functional:

\medskip

\noindent{\bf Condition (MF-FLQ-PD).} $M>0$, $G\geq 0$, $G+\bar G\geq 0$, $Q(\cdot) \geq 0$, $Q(\cdot) +\bar Q(\cdot) \geq 0$, $R(\cdot) \gg 0$, and $R(\cdot) +\bar R(\cdot) \gg 0$.

\begin{lemma}\label{4.1:Lem:IIF}
Let Condition (MF-FLQ-PD) hold. Let $(x^*(\cdot),\xi^*,u^*(\cdot))$ be an admissible triple. Denote by $(y(\cdot),z(\cdot)) \in L^2_{\mathbb F}(\Omega;C([t,T];\mathbb R^n)) \times L^2_{\mathbb F}(t,T;\mathbb R^{nd})$ the unique solution to the following MF-BSDE:
\begin{equation}\label{4.1:AdjEq}
\left\{
\begin{aligned}
& \mathrm dy(s) = -\Big\{ A(s)^\top y(s) +\bar A(s)^\top \mathbb E_t[y(s)] +C(s)^\top z(s) +\bar C(s)^\top \mathbb E_t[z(s)]\\
& \qquad +Q(s)x^*(s) +\bar Q(s) \mathbb E_t[x^*(s)] \Big\}\, \mathrm ds +\sum_{i=1}^d z_i(s)\, \mathrm dW_i(s),\quad s\in [t,T],\\
& y(T) = Gx^*(T) +\bar G\mathbb E_t[x^*(T)].
\end{aligned}
\right.
\end{equation}
Then, $(\xi^*, u^*(\cdot))$ is an optimal control of Problem (MF-FLQ) if and only if
\begin{equation}\label{4.1:Opt:IIF}
\left\{
\begin{aligned}
& M\xi^* +H^\top y(t) =0,\\
& R(\cdot) u^*(\cdot) +\bar R(\cdot) \mathbb E_t[u^*(\cdot)] +B(\cdot)^\top y(\cdot) +\bar B(\cdot)^\top \mathbb E_t[y(\cdot)] +D(\cdot)^\top z(\cdot) +\bar D(\cdot)^\top \mathbb E_t[z(\cdot)] =0.
\end{aligned}
\right.
\end{equation}
\end{lemma}

\begin{proof}
Let $(x^*(\cdot),\xi^*,u^*(\cdot))$ be an admissible triple. For any $(\xi, u(\cdot)) \in L^2_{\mathcal F_t}(\Omega;\mathbb R^n) \times L^2_{\mathbb F}(t,T;\mathbb R^m)$ and $\varepsilon \in \mathbb R$, denote by $x^\varepsilon(\cdot)$ the admissible state process under $(\xi^* +\varepsilon \xi, u^*(\cdot) +\varepsilon u(\cdot))$. Let $x_1(\cdot) = (x^\varepsilon(\cdot) -x^*(\cdot))/\varepsilon$. Then it is the unique solution to the following MF-SDE:
\begin{equation}
\left\{
\begin{aligned}
& \mathrm dx_1(s) =\Big\{ A(s)x_1(s) +\bar A(s) \mathbb E_t[x_1(s)] +B(s)u(s) +\bar B(s)\mathbb E_t[u(s)] \Big\}\, \mathrm ds\\
& +\sum_{i=1}^d \Big\{ C_i(s)x_1(s) +\bar C_i(s) \mathbb E_t[x_1(s)] +D_i(s)u(s) +\bar D_i(s)\mathbb E_t[u(s)] \Big\}\, \mathrm dW_i(s),\quad s\in [t,T],\\
& x_1(t) = H\xi.
\end{aligned}
\right.
\end{equation}
Applying It\^o's formula to $\langle x_1(\cdot),\ y(\cdot) \rangle$ leads to (the argument $s$ is suppressed for simplicity)
\begin{equation}\label{4.1:Eq1}
\begin{aligned}
& \mathbb E_t \bigg\{ \big\langle x_1(T),\ Gx^*(T) +\bar G\mathbb E_t[x^*(T)] \big\rangle +\int_t^T \big\langle x_1,\ Qx^* +\bar Q \mathbb E_t[x^*] \big\rangle\, \mathrm ds \bigg\}\\
=\ & \mathbb E_t \bigg\{ \langle H\xi,\ y(t) \rangle +\int_t^T \big\langle u,\ B^\top y +\bar B^\top \mathbb E_t[y] +D^\top z +\bar D^\top \mathbb E_t[z] \big\rangle\, \mathrm ds \bigg\}.
\end{aligned}
\end{equation}

Now, we calculate the difference:
\[
J\big( \xi^* +\varepsilon \xi, u^*(\cdot) +\varepsilon u(\cdot) \big) -J\big( \xi^*, u^*(\cdot) \big) = \varepsilon  \mathrm I_1 + \frac{\varepsilon^2}{2} \mathrm I_2,
\]
where 
\[
\begin{aligned}
\mathrm I_1 =\ & \mathbb E_t \bigg\{ \langle M\xi^*, \xi \rangle +\langle Gx^*(T),\ x_1(T) \rangle +\langle \bar G \mathbb E_t[x^*(T)],\ \mathbb E_t[x^*(T)] \rangle\\
& +\int_t^T \Big[ \langle Qx^*,\ x_1 \rangle +\langle \bar Q \mathbb E_t[x^*],\ \mathbb E_t[x_1] \rangle +\langle Ru^*,\ u \rangle +\langle \bar R \mathbb E_t[u^*],\ \mathbb E_t[u] \rangle \Big]\, \mathrm ds \bigg\}\\
=\ & \mathbb E_t \bigg\{ \langle M\xi^*,\ \xi \rangle +\big\langle Gx^*(T) +\bar G \mathbb E_t[x^*(T)],\ x_1(T) \big\rangle\\
& +\int_t^T \Big[ \big\langle Qx^* +\bar Q\mathbb E_t[x^*],\ x_1 \big\rangle +\big\langle Ru^* +\bar R \mathbb E_t[u^*],\ u \big\rangle \Big]\, \mathrm ds \bigg\}
\end{aligned}
\]
and
\[
\begin{aligned}
\mathrm I_2 =\ & \mathbb E_t \bigg\{ \langle M\xi,\ \xi \rangle +\langle Gx_1(T),\ x_1(T) \rangle +\langle \bar G \mathbb E_t[x_1(T)],\ \mathbb E_t[x_1(T)] \rangle\\
& +\int_t^T \Big[ \langle Qx_1,\ x_1 \rangle +\langle \bar Q\mathbb E_t[x_1],\ \mathbb E_t[x_1] \rangle +\langle Ru,\ u \rangle +\langle \bar R\mathbb E_t[u],\ \mathbb E_t[u] \rangle \Big]\, \mathrm ds \bigg\}\\
=\ & \mathbb E_t \bigg\{ \big\langle G\big( x_1(T) -\mathbb E_t[x_1(T)] \big),\ x_1(T) -\mathbb E_t[x_1(T)] \big\rangle +\big\langle (G+\bar G) \mathbb E_t[x_1(T)],\ \mathbb E_t[x_1(T)] \big\rangle\\
& +\langle M\xi,\ \xi \rangle +\int_t^T \Big[ \big\langle Q\big( x_1 -\mathbb E_t[x_1] \big),\ x_1 -\mathbb E_t[x_1] \big\rangle +\big\langle (Q +\bar Q) \mathbb E_t[x_1],\ E_t[x_1] \big\rangle\\
& +\big\langle R\big( u-\mathbb E_t[u] \big),\ u-\mathbb E_t[u] \big\rangle +\big\langle (R+\bar R) \mathbb E_t[u],\ \mathbb E_t[u] \big\rangle \Big]\, \mathrm ds \bigg\}.
\end{aligned}
\]
By substituting \eqref{4.1:Eq1} into the above expression of $\mathrm I_1$, we have
\[
\mathrm I_1 = \mathbb E_t \bigg\{ \big\langle \xi,\ M\xi^* +H^\top y(t) \big\rangle +\int_t^T \big\langle u,\ Ru^* +\bar R\mathbb E_t[u^*] +B^\top y +\bar B^\top \mathbb E_t[y] +D^\top z +\bar D^\top \mathbb E_t[z] \big\rangle\, \mathrm ds \bigg\}.
\]

With the previous preparations, now we prove the sufficiency. In fact, when \eqref{4.1:Opt:IIF} hold, we have  $\mathrm I_1=0$. Moreover, Condition (MF-FLQ-PD) implies that $\mathrm I_2 \geq 0$. Therefore,
\begin{equation}\label{4.1:Eq2}
J\big( \xi^* +\varepsilon \xi, u^*(\cdot) +\varepsilon u(\cdot) \big) -J\big( \xi^*, u^*(\cdot) \big) \geq 0.
\end{equation}
By the arbitrariness of $(\xi,u(\cdot))$, we prove that $(\xi^*, u^*(\cdot))$ is an optimal control of Problem (MF-FLQ).

Next, we will prove the necessity. When $(\xi^*,u^*(\cdot))$ is optimal, the inequality \eqref{4.1:Eq2} hold for all $(\xi,u(\cdot)) \in L^2_{\mathcal F_t}(\Omega;\mathbb R^n) \times L^2_{\mathbb F}(t,T;\mathbb R^m)$ and all $\varepsilon\in \mathbb R$. On the one hand, when $\varepsilon>0$, we have $\mathrm I_1 +(\varepsilon/2) \mathrm I_2 \geq 0$. Sending $\varepsilon \rightarrow 0^+$ leads to
\begin{equation}\label{4.1:Eq3}
\mathrm I_1 \geq 0, \quad \mbox{for all } (\xi,u(\cdot)) \in L^2_{\mathcal F_t}(\Omega;\mathbb R^n) \times L^2_{\mathbb F}(t,T;\mathbb R^m).
\end{equation}
On the other hand, when $\varepsilon <0$, a similar analysis yields
\begin{equation}\label{4.1:Eq4}
\mathrm I_1 \leq 0, \quad \mbox{for all } (\xi,u(\cdot)) \in L^2_{\mathcal F_t}(\Omega;\mathbb R^n) \times L^2_{\mathbb F}(t,T;\mathbb R^m).
\end{equation}
Combine \eqref{4.1:Eq3} and \eqref{4.1:Eq4} to get
\[
\mathrm I_1 = 0, \quad \mbox{for all } (\xi,u(\cdot)) \in L^2_{\mathcal F_t}(\Omega;\mathbb R^n) \times L^2_{\mathbb F}(t,T;\mathbb R^m).
\]
Finally, due to the arbitrariness of $(\xi,u(\cdot))$, the above equation implies \eqref{4.1:Opt:IIF}. The proof is finished.
\end{proof}

For convenience, we denote
\begin{equation}
\widetilde R(\cdot) = R(\cdot) +\bar R(\cdot),\quad \widetilde B(\cdot) = B(\cdot) +\bar B(\cdot),\quad
\widetilde D(\cdot) = D(\cdot) +\bar D(\cdot).
\end{equation}
We notice that, Condition (MF-FLQ-PD) implies that the matrix $M$ and the matrix-valued processes $R(\cdot)$ and $\widetilde R(\cdot)$ are invertible. Moreover, $R(\cdot)^{-1}$ and $\widetilde R(\cdot)^{-1}$ are uniformly bounded. Then, we can solve \eqref{4.1:Opt:IIF} as
\begin{equation}\label{4.1:Opt:Ini}
\xi^* = -M^{-1}H^\top y(t)
\end{equation}
and 
\begin{equation}\label{4.1:Opt:Pro}
\begin{aligned}
u^*(\cdot) =\ & -R(\cdot)^{-1} \Big\{ B(\cdot)^\top \big( y(\cdot)-\mathbb E_t[y(\cdot)] \big) +D(\cdot)^\top \big( z(\cdot)-\mathbb E_t[z(\cdot)] \big) \Big\}\\
& -\widetilde R(\cdot)^{-1} \Big\{ \widetilde B(\cdot)^\top \mathbb E_t[y(\cdot)]  +\widetilde D(\cdot)^\top \mathbb E_t[z(\cdot)] \Big\}.
\end{aligned}
\end{equation}
Now, we substitute \eqref{4.1:Opt:Ini} and \eqref{4.1:Opt:Pro} into the state equation \eqref{4.1:Sys}, and form a system together with MF-BSDE \eqref{4.1:AdjEq} (the argument $s$ is suppressed for simplicity):
\begin{equation}\label{4.1:Ham}
\left\{
\begin{aligned}
& \mathrm dx^* = \bigg\{ Ax^* +\bar A\mathbb E_t[x^*] -BR^{-1} \Big\{ B^\top \big( y-\mathbb E_t[y] \big) +D^\top \big( z-\mathbb E_t[z] \big) \Big\}\\
& \qquad -\widetilde B\widetilde R^{-1}\Big\{ \widetilde B^\top \mathbb E_t[y] +\widetilde D^\top \mathbb E_t[z] \Big\} +\alpha \bigg\}\, \mathrm ds\\
& \qquad +\sum_{i=1}^d \bigg\{ C_ix^* +\bar C_i \mathbb E_t[x^*] -D_iR^{-1} \Big\{ B^\top \big( y-\mathbb E_t[y] \big) +D^\top \big( z-\mathbb E_t[z] \big) \Big\}\\
& \qquad -\widetilde D_i\widetilde R^{-1}\Big\{ \widetilde B^\top \mathbb E_t[y] +\widetilde D^\top \mathbb E_t[z] \Big\} +\beta_i \bigg\}\, \mathrm dW_i,\quad s\in [t,T],\\
& \mathrm dy = - \Big\{ A^\top y +\bar A^\top \mathbb E_t[y] +C^\top z +\bar C^\top \mathbb E_t[z] +Qx^* +\bar Q\mathbb E_t[x^*] \Big\}\, \mathrm ds\\
& \qquad +\sum_{i=1}^d z_i\, \mathrm dW_i,\quad s\in [t,T],\\
& x^*(t) =-HM^{-1}H^\top y(t) +x_t,\qquad y(T) = Gx^*(T) +\bar G \mathbb E_t[x^*(T)].
\end{aligned}
\right.
\end{equation}
The above system is called a Hamiltonian system in the terminology of control theory, which is actually an MF-FBSDE. 

Now, we give the main result of this subsection.

\begin{proposition}\label{4.1:Prop}
Under Condition (MF-FLQ-PD), the Hamiltonian system \eqref{4.1:Ham} admits a unique solution $\theta(\cdot) =(x^*(\cdot)^\top,y(\cdot)^\top,z(\cdot)^\top)^\top \in M^2_{\mathbb F}(t,T;\mathbb R^{n(2+d)})$. Moreover, $(\xi^*, u^*(\cdot))$ defined by \eqref{4.1:Opt:Ini} and \eqref{4.1:Opt:Pro} is the unique optimal control of Problem (MF-FLQ).
\end{proposition}

\begin{proof}
Firstly, under Condition (MF-FLQ-PD), it is verified that the coefficients of MF-FBSDE \eqref{4.1:Ham} satisfy Assumptions (H3), (H4)-(i)-Case A, (H4)-(ii), and (H4)-(iii). Then, Theorem \ref{3:THM:FBSDE} shows that \eqref{4.1:Ham} admits a unique solution $\theta(\cdot) =(x^*(\cdot)^\top,y(\cdot)^\top,z(\cdot)^\top)^\top \in M^2_{\mathbb F}(t,T;\mathbb R^{n(2+d)})$. Secondly, by Lemma \ref{4.1:Lem:IIF}, the unique solvability of \eqref{4.1:Ham} implies the existence and uniqueness of optimal control of Problem (MF-FLQ). Moreover, the unique optimal control must be given by \eqref{4.1:Opt:Ini} and \eqref{4.1:Opt:Pro}.
\end{proof}

\subsection{Backward MF-LQ control Problem}

In the second LQ problem, the controlled system is given by a linear MF-BSDE:
\begin{equation}\label{4.2:Sys}
\left\{
\begin{aligned}
& \mathrm dy(s) = \Big\{ A(s)y(s) +\bar A(s) \mathbb E_t[y(s)] +B(s)z(s) +\bar B(s) \mathbb E_t [z(s)]\\
& \qquad +C(s)u(s) +\bar C(s) \mathbb E_t[u(s)] +\alpha(s) \Big\}\, \mathrm ds +\sum_{i=1}^d z_i(s)\, \mathrm dW_i(s),\quad s\in [t,T],\\
& y(T) = P\eta +\bar P \mathbb E_t[\eta] +y_T.
\end{aligned}
\right.
\end{equation}
In this subsection, we will adopt the decomposition $B(\cdot) =(B_1(\cdot), B_2(\cdot), \dots, B_d(\cdot))$ and $\bar B(\cdot) =(\bar B_1(\cdot), \bar B_2(\cdot), \dots, \bar B_d(\cdot))$. In \eqref{4.2:Sys}, we assume that $A(\cdot), \bar A(\cdot), B_i(\cdot), \bar B_i(\cdot) \in L^\infty (t,T;\mathbb R^{n\times n})$  ($i=1,2,\dots,d$), $C(\cdot), \bar C(\cdot) \in L^\infty(t,T;\mathbb R^{n\times m})$, $P, \bar P \in \mathbb R^{n\times n}$, $\alpha(\cdot) \in L^2_{\mathbb F}(\Omega;L(t,T;\mathbb R^n))$, and $y_T\in L^2_{\mathcal F_T}(\Omega;\mathbb R^n)$. The pair $(\eta, u(\cdot)) \in L^2_{\mathcal F_T}(\Omega;\mathbb R^n) \times L^2_{\mathbb F}(t,T;\mathbb R^m)$ is called an {\it admissible control}. By Proposition \ref{2:Prop:BSDE}, for any admissible control $(\eta, u(\cdot)) \in L^2_{\mathcal F_T}(\Omega;\mathbb R^n) \times L^2_{\mathbb F}(t,T;\mathbb R^m)$, MF-BSDE \eqref{4.2:Sys} admits a unique solution $(y(\cdot),z(\cdot)) \equiv (y(\cdot;\eta,u(\cdot)), z(\cdot;\eta,u(\cdot))) \in L^2_{\mathbb F}(\Omega;C([t,T];\mathbb R^n)) \times L^2_{\mathbb F}(t,T;\mathbb R^{nd})$ which is called the {\it admissible state process} under $(\eta,u(\cdot))$. Moreover, $(y(\cdot),z(\cdot),\eta,u(\cdot))$ is called an {\it admissible quadruple}. We notice that, the backward controlled system without terminal control $\eta$ has been studied by Li et al. \cite{LSX-19}. Here, we consider the general situation \eqref{4.2:Sys}. In order to evaluate admissible controls, we are also given an objective functional in a quadratic form:
\begin{equation}
\begin{aligned}
& J\big( \eta, u(\cdot) \big) = \frac 1 2 \mathbb E_t \bigg\{ \langle My(t),\ y(t) \rangle +\langle G\eta,\ \eta \rangle +\big\langle \bar G\mathbb E_t[\eta],\ \mathbb E_t[\eta] \big\rangle\\
& \qquad +\int_t^T \Big[ \langle Q(s)y(s),\ y(s) \rangle +\big\langle \bar Q(s) \mathbb E_t[y(s)],\ \mathbb E_t[y(s)] \big\rangle +\langle L(s)z(s),\ z(s) \rangle\\
& \qquad +\big\langle \bar L(s) \mathbb E_t[z(s)],\ \mathbb E_t[z(s)] \big\rangle +\langle R(s) u(s),\ u(s) \rangle +\big\langle \bar R(s) \mathbb E_t[u(s)],\ \mathbb E[u(s)] \big\rangle \Big]\, \mathrm ds \bigg\},
\end{aligned}
\end{equation}
where $M, G, \bar G \in \mathbb S^n$, $Q(\cdot), \bar Q(\cdot) \in L^\infty(t,T;\mathbb S^n)$, $L(\cdot) := \mathrm{diag} \{ L_1(\cdot), L_2(\cdot), \dots, L_d(\cdot) \}$, $\bar L(\cdot) := \mathrm{diag}\{ \bar L_1(\cdot), \bar L_2(\cdot), \dots, \bar L_d(\cdot) \}$ with $L_i(\cdot), \bar L_i(\cdot) \in L^\infty(t,T;\mathbb S^n)$ ($i=1,2,\dots,d$), and $R(\cdot), \bar R(\cdot) \in L^\infty(t,T;\mathbb S^m)$. Clearly, for any $(\eta, u(\cdot)) \in L^2_{\mathcal F_t} (\Omega;\mathbb R^n) \times L^2_{\mathbb F}(t,T;\mathbb R^m)$, $J(\eta,u(\cdot))$ is well-defined. Now, we propose a mean-field type backward LQ (MF-BLQ, for short) control problem as follows:

\medskip

\noindent{\bf Problem (MF-BLQ).} Find an admissible control $(\eta^*, u^*(\cdot)) \in L^2_{\mathcal F_T}(\Omega;\mathbb R^n) \times L^2_{\mathbb F}(t,T;\mathbb R^m)$ such that
\begin{equation}\label{4.2:Opt}
J\big( \eta^*, u^*(\cdot)\big) =\essinf_{(\eta, u(\cdot)) \in L^2_{\mathcal F_T}(\Omega;\mathbb R^n) \times L^2_{\mathbb F}(t,T;\mathbb R^m)} J\big( \eta, u(\cdot)\big).
\end{equation}
$(\eta^*, u^*(\cdot))$ satisfying \eqref{4.2:Opt} is called an {\it optimal control}, $(y^*(\cdot),z^*(\cdot)) \equiv (y(\cdot;\eta^*, u^*(\cdot)), z(\cdot;\eta^*, u^*(\cdot)))$ is called the corresponding {\it optimal state process}, and $(y^*(\cdot),z^*(\cdot),\eta^*,u^*(\cdot))$ is called an {\it optimal quadruple}.

\medskip

Similar to the previous subsection, we introduce the following uniformly positive definiteness condition for the weighting matrices in the objective functional:

\medskip

\noindent{\bf Condition (MF-BLQ-PD).} $M\geq 0$, $G>0$, $G +\bar G >0$, $Q(\cdot) \geq 0$, $Q(\cdot) +\bar Q(\cdot) \geq 0$, $L(\cdot)\geq 0$, $L(\cdot) +\bar L(\cdot) \geq 0$, $R(\cdot) \gg 0$, and $R(\cdot) +\bar R(\cdot) \gg 0$.

\medskip

\begin{lemma}\label{4.2:Lem:IIF}
Let Condition (MF-BLQ-PD) hold. Let $(y^*(\cdot),z^*(\cdot),\eta^*, u^*(\cdot))$ be an admissible quadruple. Denote by $x(\cdot) \in L^2_{\mathbb F}(\Omega;C([t,T];\mathbb R^n))$ the unique solution to the following MF-SDE:
\begin{equation}\label{4.2:AdjEq}
\left\{
\begin{aligned}
& \mathrm dx(s) = -\Big\{ A(s)^\top x(s) +\bar A(s)^\top \mathbb E_t[x(s)] +Q(s) y^*(s) +\bar Q(s) \mathbb E_t[y^*(s)] \Big\}\, \mathrm ds\\
& \quad -\sum_{i=1}^d \Big\{ B_i(s)^\top x(s) +\bar B_i(s)^\top \mathbb E_t[x(s)] +L_i(s) z_i^*(s) +\bar L_i(s) \mathbb E_t[z_i^*(s)] \Big\}\, \mathrm dW_i(s),\\
& \hskip 11cm s\in [t,T],\\
& x(t) = -My^*(t).
\end{aligned}
\right.
\end{equation}
Then, $(\eta^*, u^*(\cdot))$ is an optimal control of Problem (MF-BLQ) if and only if
\begin{equation}\label{4.2:Opt:IIF}
\left\{
\begin{aligned}
& G \eta^* +\bar G \mathbb E_t[\eta^*] -P^\top x(T) -\bar P^\top \mathbb E_t[x(T)] =0,\\
& R(\cdot) u^*(\cdot) +\bar R(\cdot) \mathbb E_t[u^*(\cdot)] +C(\cdot)^\top x(\cdot) +\bar C(\cdot)^\top \mathbb E_t[x(\cdot)] =0.
\end{aligned}
\right.
\end{equation}
\end{lemma}

The proof of the above Lemma \ref{4.2:Lem:IIF} is similar to that of Lemma \ref{4.1:Lem:IIF}. Then we omit it.

Let us denote
\[
\widetilde G = G+\bar G,\quad
\widetilde P = P+\bar P,\quad 
\widetilde R(\cdot) =R(\cdot) +\bar R(\cdot),\quad 
\widetilde C(\cdot) =C(\cdot) +\bar C(\cdot).
\]
Then, under Condition (MF-BLQ-PD), \eqref{4.2:Opt:IIF} can be rewritten in an explicit form:
\begin{equation}\label{4.2:Opt:TemPro}
\left\{
\begin{aligned}
& \eta^* = G^{-1} P^\top \big( x(T) -\mathbb E_t[x(T)] \big) +\widetilde G^{-1} \widetilde P^\top \mathbb E_t[x(T)],\\
& u^*(\cdot) = -R(\cdot)^{-1}C(\cdot)^\top \big( x(\cdot) -\mathbb E_t[x(\cdot)] \big) -\widetilde R(\cdot)^{-1} \widetilde C(\cdot)^\top \mathbb E_t[x(\cdot)].
\end{aligned}
\right.
\end{equation}
Substituting \eqref{4.2:Opt:TemPro} into the state equation \eqref{4.2:Sys} and combining with MF-SDE \eqref{4.2:AdjEq} yield the following Hamiltonian system (the argument $s$ is suppressed for simplicity):
\begin{equation}\label{4.2:Ham}
\left\{
\begin{aligned}
& \mathrm dx =- \Big\{ A^\top x +\bar A^\top \mathbb E_t[x] +Qy^* +\bar Q \mathbb E_t[y^*] \Big\}\, \mathrm ds\\
& \qquad -\sum_{i=1}^d \Big\{ B_i^\top x +\bar B_i^\top \mathbb E_t[x] +L_iz_i^* +\bar L_i\mathbb E_t[z_i^*] \Big\}\, \mathrm dW_i,\quad s\in [t,T],\\
& \mathrm dy^* = \Big\{ Ay^* +\bar A\mathbb E_t[y^*] +Bz^* +\bar B\mathbb E_t[z^*] -CR^{-1}C^\top \big(x -\mathbb E_t[x]\big) \\
& \qquad -\widetilde C \widetilde R^{-1} \widetilde C^\top \mathbb E_t[x] +\alpha \Big\}\, \mathrm ds +\sum_{i=1}^d z_i^*\, \mathrm dW_i,\quad s\in [t,T],\\
& x(t) =-My^*(t),\qquad y^*(T) = PG^{-1}P^\top \big( x(T) -\mathbb E_t[x(T)] \big) +\widetilde P\widetilde G^{-1}\widetilde P^\top \mathbb E_t[x(T)].
\end{aligned}
\right.
\end{equation}

We give the main result of this subsection.

\begin{proposition}\label{4.2:Prop}
Under Condition (MF-BLQ-PD), the Hamiltonian system \eqref{4.2:Ham} admits a unique solution $\theta(\cdot) =(x(\cdot)^\top,y^*(\cdot)^\top,z^*(\cdot)^\top)^\top \in M^2_{\mathbb F}(t,T;\mathbb R^{n(2+d)})$. Moreover, $(\eta^*, u^*(\cdot))$ defined by \eqref{4.2:Opt:TemPro} is the unique optimal control of Problem (MF-BLQ).
\end{proposition}

\begin{proof}
Under Condition (MF-BLQ-PD), we verify that the coefficients of MF-FBSDE \eqref{4.2:Ham} satisfy Assumptions (H3), (H4)-(i)-Case B, (H4)-(ii), and (H4)-(iii). Then, Theorem \ref{3:THM:FBSDE} shows that \eqref{4.2:Ham} admits a unique solution in the space $M^2_{\mathbb F}(t,T;\mathbb R^{n(2+d)})$. Similar to the proof of Proposition \ref{4.1:Prop}, the remaining of this proof is implied by Lemma \ref{4.2:Lem:IIF}.
\end{proof}

As an echo of Remark \ref{3:Rem:Sym}, we give the following remark. Due to the similarity, the details will be omitted.

\begin{remark}\label{4.2:Rem}
If the essential infimum in Problem (MF-FLQ) (resp. Problem (MF-BLQ)) is replaced with the essential supremum, and at the same time Condition (MF-FLQ-PD) (resp. Condition (MF-BLQ-PD)) is replaced with the corresponding uniformly negative definiteness condition, then we will have similar conclusions to Proposition \ref{4.1:Prop} (resp. Proposition \ref{4.2:Prop}). We notice the difference is: In the related assumptions satisfied by the Hamiltonian system arising from the new MF-LQ problem, Assumption (H4)-(iii) will be replaced by Assumption (H4)-(iii)$'$. 
\end{remark}

\begin{appendices}

\section{Proofs of Proposition \ref{2:Prop:SDE} and Proposition \ref{2:Prop:BSDE}}\label{Sec:App}

In this appendix, we devote ourselves to proving Proposition \ref{2:Prop:SDE} and Proposition \ref{2:Prop:BSDE}.

\subsection{Proof of Proposition \ref{2:Prop:SDE}}\label{Sub:App1}

We split the whole proof into three steps.

{\bf Step 1: Unique solvability.} Let $[t',T'] \subset [t,T]$ and $x_{t'} \in L^2_{\mathcal F_{t'}(\Omega;\mathbb R^n)}$. For any $x(\cdot) \in L^2_{\mathbb F}(\Omega;C([t',T'];\mathbb R^n))$, we introduce the following SDE:
\begin{equation}\label{A:SDE1}
\left\{
\begin{aligned}
& \mathrm dX(s) =b\big( s,X(s),\mathbb E_t[x(s)] \big)\, \mathrm ds +\sum_{i=1}^d \sigma_i \big( s,X(s),\mathbb E_t[x(s)] \big)\, \mathrm dW_i(s), \quad s\in [t',T'],\\
& X(t') =x_{t'}.
\end{aligned}
\right.
\end{equation}
Since the mappings $b$ and $\sigma$ are Lipschitz continuous with respect to $x'$, $b(\cdot,0,0) \in L^2_{\mathbb F}(\Omega; L([t,T];\mathbb R^n))$ and $\sigma(\cdot,0,0) \in L^2_{\mathbb F}(t,T;\mathbb R^{nd})$, then it is easy to verify that $b(\cdot,0,\mathbb E_t[x(\cdot)]) \in L^2_{\mathbb F}(\Omega;L([t',T'];\mathbb R^n))$ and $\sigma(\cdot,0,\mathbb E_t[x(\cdot)]) \in L^2_{\mathbb F}(t',T';\mathbb R^{nd})$. Therefore, the classical theory of SDEs implies that \eqref{A:SDE1} admits a unique solution $X(\cdot) \in L^2_{\mathbb F}(\Omega;C([t',T'];\mathbb R^n))$. Due to the arbitrariness of $x(\cdot)$, we have actually constructed a mapping $\mathscr T_F: L^2_{\mathbb F}(\Omega;C([t',T'];\mathbb R^n)) \rightarrow L^2_{\mathbb F}(\Omega;C([t',T'];\mathbb R^n))$:
\[
\mathscr T_F(x(\cdot)) = X(\cdot).
\]

Let $x(\cdot), \bar x(\cdot) \in L^2_{\mathbb F}(\Omega;C([t',T'];\mathbb R^n))$ and $X(\cdot) =\mathscr T_F(x(\cdot))$, $\bar X(\cdot) =\mathscr T_F(\bar x(\cdot))$. Denote $\widehat x(\cdot) =x(\cdot) -\bar x(\cdot)$ and $\widehat X(\cdot) =X(\cdot) -\bar X(\cdot)$. Then the continuous dependence property of the solution on the coefficients in the classical theory of SDEs shows
\begin{equation}\label{A:Eq1}
\begin{aligned}
\mathbb E\bigg[ \sup_{s\in [t',T']}\big| \widehat X(s) \big|^2 \bigg] \leq\ & K_F \mathbb E\bigg\{ \bigg( \int_{t'}^{T'} \Big| b\big( s,\bar X(s), \mathbb E_t[x(s)] \big) -b\big( s,\bar X(s), \mathbb E_t[\bar x(s)] \big) \Big|\, \mathrm ds \bigg)^2\\
& +\int_{t'}^{T'} \Big| \sigma\big( s,\bar X(s), \mathbb E_t[x(s)] \big) -\sigma\big( s,\bar X(s), \mathbb E_t[\bar x(s)] \big) \Big|^2\, \mathrm ds \bigg\},
\end{aligned}
\end{equation}
where $K_F:=K_F(T-t,L)>0$ is a constant depending on $(T-t)$ and the Lipschitz constant of the mappings $b$ and $\sigma$ with respect to $x$. Moreover, the Lipschitz continuity of $b$ and $\sigma$ with respect ot $x'$ was also used to yield
\[
\begin{aligned}
\mathbb E\bigg[ \sup_{s\in [t',T']}\big| \widehat X(s) \big|^2 \bigg] \leq\ & K_F L^2 \mathbb E\bigg\{ \bigg( \int_{t'}^{T'} \mathbb E_t \big[|\widehat x(s)| \big]\, \mathrm ds \bigg)^2 +\int_{t'}^{T'} \Big( \mathbb E_t\big[ |\widehat x(s)| \big] \Big)^2\, \mathrm ds \bigg\}\\
\leq\ & 2K_F L^2 (T'-t') \max \big\{ (T'-t'), 1 \big\} \mathbb E\bigg[ \sup_{s\in [t',T']}\big| \widehat x(s) \big|^2 \bigg].
\end{aligned}
\]
Let
\begin{equation}\label{A:deltaF}
\delta_F = \min \bigg\{ 1,\ \frac{1}{8K_F L^2} \bigg\}.
\end{equation}
Then, when $T'-t'\leq \delta_F$, the above inequality implies 
\[
\big\Vert \widehat X(\cdot) \big\Vert_{L^2_{\mathbb F}(\Omega;C([t',T'];\mathbb R^n))} \leq \frac 1 2 \big\Vert \widehat x(\cdot) \big\Vert_{L^2_{\mathbb F}(\Omega;C([t',T'];\mathbb R^n))},
\]
i.e., the mapping $\mathscr T_F$ is a contraction. Due to Banach's contraction mapping theorem, it admits a unique fixed point which is the unique solution to the following MF-SDE:
\begin{equation}
\left\{
\begin{aligned}
& \mathrm dx(s) =b\big( s,x(s),\mathbb E_t[x(s)] \big)\, \mathrm ds +\sum_{i=1}^d \sigma_i \big( s,x(s),\mathbb E_t[x(s)] \big)\, \mathrm dW_i(s), \quad s\in [t',T'],\\
& x(t') =x_{t'}.
\end{aligned}
\right.
\end{equation}

Now, we turn our attention to MF-SDE \eqref{2:SDE}. Based on the previous study, we divide the whole interval into 
\begin{equation}
N_F =\inf \bigg\{ \bar N \in \mathbb N\, \bigg|\ \bar N \geq \frac{T-t}{\delta_F} \bigg\}
\end{equation}
sub-intervals with $\delta_F$ defined by \eqref{A:deltaF} as the step length. Then, we derive that MF-SDE \eqref{2:SDE} admits a unique solution $x(\cdot)\in L^2_{\mathbb F}(\Omega;C([t,T];\mathbb R^n))$.

{\bf Step 2: Estimate \eqref{2:SDE:Est1}.} For any $s\in [t,T]$, the classical estimate of SDEs leads to
\[
\begin{aligned}
& \mathbb E_t \bigg[ \sup_{r\in [t,s]} |x(r)|^2 \bigg]\\
\leq\ & K_F \bigg\{ |x_t|^2 +\mathbb E_t \bigg[ \bigg( \int_t^s \Big| b\big(r,0,\mathbb E_t[x(r)]\big) \Big|\, \mathrm dr \bigg)^2 +\int_t^s \Big| \sigma\big(r,0,\mathbb E_t[x(r)]\big) \Big|^2\, \mathrm dr \bigg] \bigg\}\\
\leq\ & M_F +2K_F L^2 \mathbb E_t \bigg[ \bigg( \int_t^s \mathbb E_t[|x(r)|]\, \mathrm dr \bigg)^2 +\int_t^s \Big( \mathbb E_t [|x(r)|] \Big)^2\, \mathrm dr \bigg],
\end{aligned}
\]
where $K_F$ is the same constant as in \eqref{A:Eq1} and
\begin{equation}
M_F := 2K_F \bigg\{ |x_t|^2 +\mathbb E_t \bigg[ \bigg( \int_t^T \big|b(r,0,0)\big|\, \mathrm dr \bigg)^2 +\int_t^T \big| \sigma(r,0,0) \big|^2\, \mathrm dr \bigg] \bigg\}.
\end{equation}
We continue to derive
\[
\mathbb E_t \bigg[ \sup_{r\in [t,s]} |x(r)|^2 \bigg]
\leq M_F +2K_F L^2(T-t+1) \int_t^s \mathbb E_t \bigg[\sup_{u\in [t,r]} |x(u)|^2\bigg]\, \mathrm dr.
\]
Then, Gronwall's inequality leads to the desired estimate \eqref{2:SDE:Est1}.

{\bf Step 3: Estimate \eqref{2:SDE:Est2}.} Denote $\widehat x_t :=x_t -\bar x_t$ and
\[
\begin{aligned}
& \widehat b(s,x,x') := b\big( s, x+\bar x(s), x'+\mathbb E_t[\bar x(s)] \big) -\bar b\big( s, \bar x(s), \mathbb E_t[\bar x(s)] \big),\\
& \widehat \sigma(s,x,x') := \sigma\big( s, x+\bar x(s), x'+\mathbb E_t[\bar x(s)] \big) -\bar \sigma\big( s, \bar x(s), \mathbb E_t[\bar x(s)] \big)
\end{aligned}
\]
for all $(s,\omega,x,x') \in [t,T]\times \Omega\times \mathbb R^n \times \mathbb R^n$. It is easy to verify that the new defined coefficients $(\widehat x_t, \widehat b, \widehat \sigma)$ also satisfy Assumption (H1) with the same Lipschitz constant. Therefore, by the proved Step 1, the following MF-SDE
\begin{equation}\label{A:SDE2}
\left\{
\begin{aligned}
& \mathrm d\widehat x(s) =\widehat b\big( s, \widehat x(s), \mathbb E_t[\widehat x(s)] \big)\, \mathrm ds +\sum_{i=1}^d \widehat \sigma_i\big( s, \widehat x(s), \mathbb E_t[\widehat x(s)] \big)\, \mathrm dW_i(s),\quad s\in [t,T],\\
& \widehat x(t) =\widehat x_t
\end{aligned}
\right.
\end{equation}
admits a unique solution. Moreover, we directly verify that the unique solution happens to be $\widehat x(\cdot) = x(\cdot) -\bar x(\cdot)$. Finally, by the proved Step 2, the estimate \eqref{2:SDE:Est1} is applied to MF-SDE \eqref{A:SDE2} to yield \eqref{2:SDE:Est2}. The proof is finished.

\subsection{Proof of Proposition \ref{2:Prop:BSDE}}\label{Sub:App2}

Similar to the proof of Proposition \ref{2:Prop:SDE}, here we also split the whole proof into three steps.

{\bf Step 1: Unique solvability.} This step is similar to Step 1 in the proof of Proposition \ref{2:Prop:SDE}. Then, we would like to omit it.

{\bf Step 2: Estimate \eqref{2:BSDE:Est1}.} For any $[t',T'] \subset [t,T]$, by the classical estimate of BSDEs, we have
\[
\begin{aligned}
& \mathbb E_t \bigg[ \sup_{s\in [t',T']}|y(s)|^2 +\int_{t'}^{T'} |z(s)|^2\, \mathrm ds \bigg]\\
\leq\ & K_B \mathbb E_t \bigg\{ \big|y(T')\big|^2 +\bigg( \int_{t'}^{T'} \Big| g\big( s,0,\mathbb E_t[y(s)],0,\mathbb E_t[z(s)] \big) \Big|\, \mathrm ds \bigg)^2 \bigg\}\\
\leq\ & 3K_B \mathbb E_t \bigg\{ \big|y(T')\big|^2 +\bigg( \int_{t'}^{T'} \Big| g\big( s,0,0,0,0 \big) \Big|\, \mathrm ds \bigg)^2\\
& +L^2(T'-t')^2 \sup_{s\in [t',T']} |y(s)|^2 +L^2(T'-t')\int_{t'}^{T'} |z(s)|^2\, \mathrm ds \bigg\},
\end{aligned}
\]
where $K_B: = K_B(T-t,L)>0$ is a constant depending on $(T-t)$ and the Lipschitz constant of $g$ with respect to $(y,z)$. Let
\begin{equation}\label{A:deltaB}
\delta_B := \min \bigg\{1,\ \frac{1}{6K_B L^2}\bigg\}.
\end{equation}
Then, when $T'-t'\leq \delta_B$, the above inequality implies
\begin{equation}\label{A:Eq2}
\mathbb E_t \bigg[ \sup_{s\in [t',T']}|y(s)|^2 +\int_{t'}^{T'} |z(s)|^2\, \mathrm ds \bigg] \leq 6K_B \mathbb E_t \bigg\{ \big|y(T')\big|^2 +\bigg( \int_{t'}^{T'} \Big| g\big( s,0,0,0,0 \big) \Big|\, \mathrm ds \bigg)^2 \bigg\}.
\end{equation}

Now, we divide the whole interval $[t,T]$ into 
\begin{equation}
N_B =\inf \bigg\{ \bar N \in \mathbb N\, \bigg|\ \bar N \geq \frac{T-t}{\delta_B} \bigg\}
\end{equation}
sub-intervals with $\delta_B$ defined by \eqref{A:deltaB} as the step length. Denote the split points by $t =t_0 <t_1<t_2 <\cdots <t_{N_B-1} <t_{N_B} =T$. Then, for each $k=0,1,\dots,N_B-1$, from \eqref{A:Eq2}, we have
\begin{equation}\label{A:Eq3}
\mathbb E_t \bigg[ \sup_{s\in [t_k,t_{k+1}]}|y(s)|^2 +\int_{t_k}^{t_{k+1}} |z(s)|^2\, \mathrm ds \bigg] \leq 6K_B \mathbb E_t \bigg\{ \big|y(t_{k+1})\big|^2 +\bigg( \int_{t_k}^{t_{k+1}} \Big| g\big( s,0,0,0,0 \big) \Big|\, \mathrm ds \bigg)^2 \bigg\}.
\end{equation}
Without loss of generality, we can assume that $6K_B \geq 1$ (otherwise, we can use $\widetilde K_B := \max\{ K_B, 1/6\}$ instead of $K_B$). For $k=N_B-1, N_B-2,\dots,0$, by some recursive but straightforward calculations from \eqref{A:Eq3}, we have
\begin{equation}\label{A:Eq4}
\begin{aligned}
& \mathbb E_t \bigg[ \sup_{s\in [t_k,t_{k+1}]}|y(s)|^2 +\int_{t_k}^{t_{k+1}} |z(s)|^2\, \mathrm ds \bigg]\\
\leq\ & (6K_B)^{N_B-k} \mathbb E_t \bigg\{ |y_T|^2 +\bigg( \int_{t_k}^T \Big|g(s,0,0,0,0)\Big|\, \mathrm ds \bigg)^2 \bigg\}\\
\leq\ & (6K_B)^{N_B-k} \mathbb E_t \bigg\{ |y_T|^2 +\bigg( \int_t^T \Big|g(s,0,0,0,0)\Big|\, \mathrm ds \bigg)^2 \bigg\}.
\end{aligned}
\end{equation}
Since
\[
\sup_{s\in [t,T]} |y(s)|^2 \leq \sum_{k=0}^{N-1} \sup_{s\in [t_k,t_{k+1}]} |y(s)|^2 \quad \mbox{and} \quad
\int_t^T |z(s)|^2\, \mathrm ds = \sum_{k=0}^{N-1} \int_{t_k}^{t_{k+1}} |z(s)|^2\, \mathrm ds,
\]
we sum up \eqref{A:Eq4} from $k=0$ to $k=N_B-1$ to get
\begin{equation}
\mathbb E_t \bigg[ \sup_{s\in [t,T]}|y(s)|^2 +\int_t^T |z(s)|^2\, \mathrm ds \bigg] \leq \bigg( \sum_{k=0}^{N_B-1} (6K_B)^{N_B-k} \bigg) \mathbb E_t\bigg[ |y_T|^2 +\bigg( \int_t^T \big| g\big( s,0,0,0,0 \big)  \big|\, \mathrm ds \bigg)^2 \bigg].
\end{equation}
The estimate \eqref{2:BSDE:Est1} for MF-BSDEs is proved.

{\bf Step 3: Estimate \eqref{2:BSDE:Est2}.} This step is similar to Step 3 in the proof of Proposition \ref{2:Prop:SDE}. We omit it.

\end{appendices}

\end{document}